\documentclass{amsart}

\usepackage{latexsym,amssymb,amsmath, amsfonts}
\usepackage{amsthm}
\usepackage{amscd}
\usepackage[shortlabels]{enumitem}
\usepackage{a4wide}

\numberwithin{equation}{section}
\usepackage[usenames,dvipsnames,svgnames,table]{xcolor}
\usepackage{hyperref}
\hypersetup{colorlinks=true,
	linkcolor=blue,
	urlcolor=blue,
	citecolor=red} 
\newcommand{\R}{\mathbb{R}}
\newcommand{\C}{\mathbb{C}}
\newcommand{\Z}{\mathbb{Z}}

\renewcommand{\le}{\leqslant}
\renewcommand{\ge}{\geqslant}
\renewcommand{\leq}{\leqslant}
\renewcommand{\geq}{\geqslant}

\newcommand{\be}{\begin{equation}}
\newcommand{\en}{\end{equation}}
\newcommand{\ee}{\end{equation}}

\newcommand{\bt}{\begin{theorem}}
\newcommand{\et}{\end{theorem}}
\newcommand{\bp}{\begin{proof}}
\newcommand{\ep}{\end{proof}}
\newcommand{\bc}{\begin{cor}}
\newcommand{\ec}{\end{cor}}
\newcommand{\bl}{\begin{lemma}}
\newcommand{\el}{\end{lemma}}
\newcommand{\bprop}{\begin{prop}}
\newcommand{\eprop}{\end{prop}}

\newcommand{\N}{\mathbb{N}}

\newtheorem{innercustomthm}{Theorem}

\newtheorem{theorem}{Theorem}[section]
\newtheorem{remark}{Remark}
\newtheorem{lemma}[theorem]{Lemma}

\newtheorem{proposition}[theorem]{Proposition}

\numberwithin{theorem}{section} \numberwithin{definition}{section}

\newcommand{\RNum}[1]{\uppercase\expandafter{\romannumeral #1\relax}}
\def\R{\mathbb{R}}
\def\N{\mathbb{N}}

\def\Z{\mathbb{Z}}
\def\C{\mathbb{C}}

\newcommand{\Rn}{\mathbb{R}^n}
\newcommand{\D}{\mathcal{D}}
\newcommand{\vertiii}[1]{{\left\vert\kern-0.25ex\left\vert\kern-0.25ex\left\vert #1 
		\right\vert\kern-0.25ex\right\vert\kern-0.25ex\right\vert}}

\theoremstyle{definition}

\author[A. Mu\~noz]{Alexander Mu\~noz}
\address{IMECC-UNICAMP, Rua S\'ergio Buarque de Holanda, 651, 13083-859, Cam\-pi\-nas-SP, Bra\-zil}
\email{alexd@ime.unicamp.br}
\author[A. Pastor]{Ademir Pastor}
\address{IMECC-UNICAMP, Rua S\'ergio Buarque de Holanda, 651, 13083-859, Cam\-pi\-nas-SP, Bra\-zil}
\email{apastor@ime.unicamp.br}

\thanks{}


\date{}
\title[Persistence properties and applications to dispersive blow up]{Local well-posedness in weighted Sobolev spaces for nonlinear dispersive equations with applications to dispersive blow up}

\begin{document}

\begin{abstract} 
In the first part of this work we study the local well-posedness of dispersive equations in the weighted spaces $H^s(\R)\cap L^2(|x|^{2b}dx)$. We then apply our results for several dispersive models such as the Hirota-Satsuma system, the OST equation, the Kawahara equation and a fifth-order model. Using these local results, the second part of this work is devoted to obtain results related to dispersive blow up of the Kawahara equation and Hirota-Satsuma system. 
\end{abstract}

\maketitle

\section{Introduction}\label{intro}

A fundamental aspect in theory of dispersive equations is the study of well-posedness. Following Kato \cite{KATO}, we say that the initial value problem (IVP) 
\begin{equation}
	\label{wpsintro}	
	\begin{cases}
		\partial_t u (x,t)= f(u),\quad x\in \R^n,\ t\in \R \\
	u(x,0)=u_0(x) 
	\end{cases}
\end{equation}
is \textit{locally well-posed} in the Banach space $Y$ if the next two conditions are satisfied:
\begin{enumerate}
	\item For each initial data $u_0\in Y$ there exist $T>0$ with a unique solution $u$ in the space $C([0,T];Y)$. 
	\item The data-solution map $u_0\mapsto u$ is continuous from $Y$ to $C([0,T];Y)$.
\end{enumerate}
In case $T$ can be selected arbitrarily large, we say the IVP is \textit{globally well-posed} in $Y$. 
It is worth to emphasize that condition (1) above is actually requiring two properties: the existence of a unique solution and its persistence in the functional space $Y$ along time. The persistence property is one of the main concerns of this work. 

The IVP associated to several dispersive equations has been considered extensively in the literature. Classical methods as the contraction principle have been employed to obtain local  well-posedness in functional spaces measuring regularity of the solutions (see for instance \cite{Caz}, \cite{KATO}, \cite{KPV} and the references therein). In \cite{KATO}, when studying the so-called Korteweg-de Vries (KdV) equation, Kato also considered spaces that, in addition to smoothness, also measure the decay of the solutions. Among the possibilities, persistence in the spaces $Z_{s,b}:=H^s(\R^n)\cap L^2(|x|^{2b}dx)$ plays an important role. The relation between decay and regularity displayed by the Fourier transform suggest the study of the persistence in such spaces. Several classical results support the existence a natural bound between the two spaces involved in the definition of $Z_{s,b}$. 

 In the past years, new techniques based on the Besov or Bourgain spaces, have been used to address the IVP associated to many dispersive equations in low regularity spaces; unfortunately, the relation between decay and regularity under these new technologies is not well understood yet by the authors. 

Earlier works dealing with persistence in the spaces $Z_{s,b}$ are based on formulas that interchange weights with the group associated to the linear part of the underlying equation. In \cite{HTN1}, \cite{HTN2} and \cite{HTN3}, based on the commutative properties of the operators $\Gamma_j=x_j+2it\partial_j$, the authors used the equality $$x^\alpha e^{it\Delta}u_0=e^{it\Delta}\Gamma^\alpha u_0, \ \ \ \alpha\in\N^n;$$ and some calculus  inequalities for the operators $\Gamma_j$ to show that if $u_0\in Z_{m,k/2}$, with $m,k$ integers, then the IVP associated with the Schr\"odinger equation,
\begin{equation}
\begin{cases}
	i\partial_tu+\Delta u+\mu |u|^{a-1}u=0, \qquad a>1,\\
	u(x,0)=u_0(x),
\end{cases}
\end{equation}
has   a unique solution $$u\in C([0,T];Z_{m,k/2})\cap L^q([0,T]; L^p_k(\R^n)\cap L^p(|x|^{k}dx)), $$ for appropriate $m$ and $k$. Here $(p,q)$ is some admissible  pair. 
This result for indices $m, k$ not necessarily integers was extended in  \cite[Theorem 1]{NAHASPONCE}.

In a similar fashion, the work of Kato \cite{KATO} for the KdV equation in the spaces $Z_{2r,r}$, $r\ge 1$ integer,   was extended  in \cite{NAHAS} to non-integer indices. Despite the lack of a closed expression for the oscillatory integral defining the linear group $\{U(t)\}_{t\in\R}$ of the KdV, in \cite{FLP}, the authors managed to obtain a simplified (and stronger) version of the results in \cite{NAHAS}. The new key ingredient was the formula
\begin{equation}\label{trocaflp}
	|x|^\alpha U(t)u_0(x)=U(t)\left\{(|x|^\alpha u_0)(x)+\left[\Phi_{t,\alpha}(\widehat{u_0}(\xi))\right]^\vee(x)\right\},
\end{equation} 
for $\alpha\in(0,1)$ and where $\Phi_{t,\alpha}$ is a remainder. To prove formula \eqref{trocaflp}, the authors used the following version of Stein's derivative: \begin{equation}\label{steinderiv2}D_\alpha f(x)=\lim\limits_{\epsilon\to0}\frac{1}{c_\alpha}\int_{|y|\ge\epsilon}\frac{f(x+y)-f(x)}{|y|^{n+\alpha}}\ dy.
\end{equation}
The main advantage of this version relies on the fact that for suitable functions $f$, it follows that $\widehat{D_\alpha(f)}=|\xi|^\alpha\widehat{f}$. This allowed the authors to recover the unitary group after a convenient application of a Leibniz-type rule for Stein derivatives. 
The remainder can be estimated in terms of the regularity of $u_0$.

The Benjamin-Ono equation
$$
\partial_t u+\mathcal{H}\partial_x^2u+uu_x=0,
$$
where $\mathcal{H}$ denotes the Hilbert transform has also been studied in the spaces $Z_{s,b}$. For integers $s$ and $b$, persistence in these spaces was first studied in \cite{Iorio}. For non-integer indices the persistence properties were established in \cite{FonPonce}. For the study of the IVP associated with other dispersive equations we refer the reader to \cite{BJMZK}, \cite{BJM0}, \cite{CP}, \cite{CP1}, \cite{FLP1}, \cite{pachon}, \cite{J}, \cite{Riano} and references therein.

\subsection{Main Results} Our first concern is to study decay properties of solutions to linear problems for several dispersive equations. More precisely, we are interested in discussing the problem
\begin{equation} \label{linearivp}
	\begin{cases}
		\partial_tu+Lu=0, \quad x\in\R^n,\;t\in\R,\\
		u(0)=u_0,
	\end{cases}
\end{equation}
where $L$ is a linear operator satisfying $\widehat{Lf}(\xi)=i\phi(\xi)\widehat{f}(\xi)$ for some continuous real-valued function $\phi$. Via Fourier transform, the solution of \eqref{linearivp} is given by
$$
U(t)u_0(x)=u(x,t)=(e^{-it\phi(\xi)}\widehat{u}_0)^\vee(x),
$$
where $\{U(t)\}_{t\in\R}$ is the associated linear group.

We shall assume that the phase $\phi:\R^n\to\R$ satisfies the following conditions:

\begin{equation}\tag{A}
	\begin{split}
		&\text{ There exists a continuous function } g:\R^n\to\R \text{ such that } g>0, \text{ except maybe at }0,\\&\text{ and for all }
		x,y\in \Rn\text{ with }|x-y|\le |x|\text{ we have } |\phi(x)-\phi(y)|\le g(x) |x-y|. \label{AA}
	\end{split}
\end{equation}

\begin{equation}\tag{B}
	\begin{split}
		&\text{There exists } C>0 \text{ such that for all }  x,y\in\Rn \text{ satisfying } |x-y|\ge|x| \text{ we have}\\& |\phi(x)-\phi(y)|\le C |x-y|^a, \text{for some }a\ge 1.  \label{BB}
	\end{split}
\end{equation}

Note that by taking $x=0$ in condition \eqref{BB} we deduce that $|\phi(y)|\leq C(1+|y|^a)$ for any $y\in\R^n$. In particular, from Stone's theorem one can see that $L$ generates a unitary group in $H^s(\R^n)$, for any $s\in\R$.

\vskip.2cm

\noindent {\bf Examples.} Some examples of phase functions satisfying \eqref{AA} and \eqref{BB} are given below.
Assume $k\in\Z^+$ and $i\in\{1,\dots,n\}$.
\begin{enumerate}
	\item Let $\phi_1:\R\to\R$ be given by $\phi_1(x)=x^k$. In this case we may take  $g(x):=C_k|x|^{k-1}$ and $a=k$. In the particular case $k=3$ we see that $\phi_1(x)=x^3$ is the phase function associated with the linear KdV equation.
	\item Let $\phi_2:\R^n\to\R$ be given by $\phi_2(x)=|x|^k$. Here we may take again $g(x)=C_k|x|^{k-1}$ and $a=k$. Note that for $k=2$ we obtain $\phi_2(x)=|x|^2$ which is the phase function associated with the linear Schr\"odinger equation.
	\item Denote by $\widehat{x_i}:=(x_1,\dots,x_{i-1},x_{i+1},\dots,x_n).$ The functions $\phi_3^i:\R^n\to\R$ defined by $\phi^i_3(x):=x_i|\widehat{x_i}|^2$, also satisfy \eqref{AA} and \eqref{BB}. In this case $g(x)=C_k|x|^2$ and $a=3$.
	\item Define $\phi_4^i:\R^n\to\R$ as $\phi_4^i(x)=x_i^k$. Then $\phi_4^i$ also satisfies \eqref{AA} and \eqref{BB} with $g(x)=C_k|x|^{k-1}$ and $a=k$. Alternatively, we may also take $g(x)=C_k|x_i|^{k-1}$ (see \cite{BJMZK}). By taking $\phi(x)=\phi_4^1(x)+\phi_3^1(x)=x_1^3 + x_1|\widehat{x}_1|^2$, we see that $\phi$ is the phase function associated with the $n$-dimensional Zakharov-Kuznetsov equation,
	$$
	\partial_t u+\partial_x\Delta u+u\partial_xu=0.
	$$
	\item More generally, for $\beta\in \N^n$, by taking $\phi_5:\R^n\to\R$ as $\phi_5(x)=x^\beta$ we obtain that it satisfies \eqref{AA} and \eqref{BB} with $g(x)=C_\beta|x|^{|\beta|-1}$ and $a=|\beta|$.
\end{enumerate}

Our main theorem concerning the persitence property of the solutions of \eqref{linearivp} reads as follows.

\begin{theorem}\label{linearteo}
	Let $p\in \Z^+$ and assume that $\phi_1,\dots, \phi_p$ satisfy the condition \eqref{AA} and \eqref{BB} with $g_i(x)\le C_i (1+|x|^{k_i})$, for some $k_i\in\Z^+$ and $C_i>0$, $i=1\ldots,p$. Set $$\Phi(\xi):=\displaystyle\sum_{i=1}^{p}\phi_i(\xi)$$ and $K:=\max \{k_i, i=1,\ldots,p\}$. Let $L$ be the linear operator defined by $Lf=\left(i\Phi(\xi)\widehat{f} \right)^\vee$ and assume $0<s<K$.
	If $u\in C([-T,T],H^s(\R^n))$ is the solution of the IVP
	\begin{equation}\label{lpvi}
		\begin{cases}
			\partial_t u +Lu=0, \quad x\in\R^n, \ t\in \R,  \\
			u(0)=u_0\in Z_{s,b}:=H^s(\R^n)\cap L^2(|x|^{2b}dx),
		\end{cases}
	\end{equation}
	with  $0<b\le s/K$, then $u$ satisfy the inequality
	\begin{equation}\label{trocas}
		\||x|^b u(t) \|_{L^2}\le C\left\{(1+|t|)\|u_0\|_{s,2} + \||x|^bu_0\|_{L^2}  \right\}
	\end{equation}
	where $\|\cdot\|_{s,2}$ denotes the norm in $H^s(\R^n)$ and $C$ depends on $K$, $p$, $s$ and $n$.
\end{theorem}

Note that \eqref{trocas} indeed establishes that the solution of \eqref{lpvi} persists in $Z_{s,b}$ for any time interval.
A result similar to Theorem \ref{linearteo} was obtained in \cite[Theorem 1.11]{XN}. The authors considered a  phase function given by $$\Phi(\xi)=\sum_{j=1}^{p}C_j\xi^{\beta_j}, \ \ \xi\in\R^n, \ \ \beta_j\in(\Z^+)^n.$$ 
and established the inequality 
$$\||x|^bu(t)\|_{L^2}\le C\||x|^bu_0\|_{L^2}+A(\|u_0\|_{H^{a(b)}}),$$ 
where $b\ge1$, $A$ is a non negative continuous function and $a(b):=\max\limits_{1,\dots,p}(|\beta_j|-1)b$. Their proof relies on estimates based on the differential equation itself. On the other hand, our approach to prove Theorem \ref{linearteo} follows the ideas in \cite{NAHASPONCE} and relies on estimates based on Stein's derivative $\D^b$ of the phase function (see \eqref{steinn}). In consequence, we are able to include weights with $0<b<1$ and establish the same interpolation inequality with $A(x)=(1+|t|)x$ and $a(b)=Kb$.

We point out that Theorem \ref{linearteo} may be seen as an alternative to \eqref{trocaflp} in the sense that it interchanges  weights with the semigroup  but also accepts several dimensions and a wide variety of phase functions. As we already highlighted, phases such as the one defining the linear part of the Zakharov-Kusnetzov equation or multivariate polynomials are included. On the other hand, in contrast with \eqref{trocaflp} we loose the punctual identity.  A  disadvantage of \eqref{trocas} compared to \eqref{trocaflp}, is the impossibility of using Strichartz type estimates once Theorem \ref{linearteo} has been applied. This prevents the application of the theory developed here in the context of estimates that do not rely on the $L^2$-based Sobolev spaces. An example of this situation is the nonlinear Schr\"odinger equation in which the inequalities used to prove local well-posedness are based on the spaces $L^p_b(\R^n)$.

The main tool to prove Theorem \ref{linearteo} is the estimate presented in Lemma \ref{l1}, which in turn is based on previous articles that faced persistence properties for particular equations such as \cite{BJM0}, \cite{FLP} and \cite{NAHASPONCE}. Some other works in which related computations have been done are \cite{pachon} and \cite{J}. 
In \cite{NAHASPONCE}, the authors dealt with the Schr\"odinger equation; using the Stein derivative defined as
\begin{equation}\label{steinn}
	\D^b(f)(x):=\displaystyle\left( \int_{\Rn}{\frac{|f(y)-f(x)|^2}{|x-y|^{n+2b}}dy} \right)^{\frac{1}{2}},
\end{equation}
they estimated $\mathcal{D}^b(e^{it|\xi|^2})(x)$ by exploding the radial behavior of the integral $$
\int_{\R^n}{\frac{|e^{i(-2\sqrt{t}x\cdot y + |y|^2)}-1|^2}{|y|^{n+2b}}dy}.
$$ This estimate was later extended in \cite{BJM0} for $\mathcal{D}^b(e^{it\xi^3})(x)$ when dealing with the Ostrovsky equation. We follow these ideas to generalize it for phase function satisfying \eqref{AA} and \eqref{BB}. 
 
 It is worth mentioning that the modulus present in the definition of $\D{^b}$  generates cancellation of oscillations when $f$ is of the form $e^{it\phi}$, preventing estimate \eqref{trocas} to be in terms of the semigroup associated to $\phi$, in contrast with \eqref{trocaflp}. This issue shrinks optimal applications of Theorem \ref{linearteo} for some nonlinear equations, in which the problem can be resolved using regularization via Sobolev embedding but that might imply extra constrains in the regularity index $s$ that may not match the best local well-posedness result available.

As a direct application of Theorem \ref{linearteo} we prove local well-posedness results in weighted spaces for several physical models. The first model we are interested in is the Hirota-Satsuma system 
\begin{equation}\label{system}
	\begin{cases}
		\partial_t u -a(\partial_x^3 u+6u\partial_x u)=2rv\partial_x v, \quad (x,t)\in\R\times\R,\\
		\partial_t v + \partial_x^3 v +3u\partial_x v =0, \\
		u(0,x)=u_0(x), \quad  v(0,x)=v_0(x),
	\end{cases}
\end{equation}
where $u$ and $v$ are real-valued functions of the variables $x,t\in \R$ and $a,r$ are real constants. The system \eqref{system} was derived in \cite{HS} and describes interactions of two long waves with different dispersion relations. Concerning local well-posedness of \eqref{system} in the standard Sobolev spaces, the following result was proved in \cite{HSS}.

\begin{innercustomthm}\textrm{$($\cite[Theorem 2.1]{HSS}$)$}\label{borys}
	Let $a\neq0$ and $s>3/4$. Then for any $u_0,v_0\in H^s(\R)$, there exists $T=T(\|u_0\|_{s,2},\|v_0\|_{s,2})>0$ and a unique solution $(u,v)$ of problem \eqref{system} such that
	\begin{equation}\label{ints}
		\begin{array}{lll}
			u,v\in C([-T,T]; H^s(\R)), & \partial_x u, \partial_x v \in L^4_T L^\infty_x,&
			D^s_x\partial_x u, D^s_x\partial_x v \in L^\infty_xL^2_T, \\ u, v \in L^2_xL^\infty_T, &\partial_x u, \partial_x v \in L^\infty_xL^2_T.
		\end{array}
	\end{equation}
Moreover, for any $T'\in (0,T)$ there exist neighborhoods $V$ of $u_0$ in $H^s(\R)$ and $V'$ of $v_0$ in $H^s(\R)$ such that the map $(\tilde{u_0},\tilde{v_0})\mapsto(\tilde{u},\tilde{v})$ from $V\times V'$ into the class defined by \eqref{ints}, with $T'$ instead of $T$, is Lipschitz. 
\end{innercustomthm}

By performing a natural modification of the Banach space used to establish the proof of Theorem \ref{borys} via contraction principle, we employ \eqref{trocas} to prove:

\begin{theorem}\label{persistHSS}
	Assume, in addition to the hypotheses in Theorem \ref{borys}, that $u_0, v_0\in L^2(|x|^{2b}dx)$ with $b\le s/2$ and $s<2$. Then there exists $T=T(\|u_0\|_{Z_{s,r}},\|v_0\|_{Z_{s,r}})>0$ and a unique solution $(u,v)$ of \eqref{system} such that $u,v$ are in the class defined by \eqref{ints} with $Z_{s,b}$ instead of $H^s(\R)$.
	
	Moreover, for any $T'\in (0,T)$ there exist neighborhoods $V$ of $u_0$ in $Z_{s,b}(\R)$ and $V'$ of $v_0$ in $Z_{s,b}$ such that the map $(\widetilde{u}_0,\widetilde{v}_0)\mapsto(\widetilde{u},\widetilde{v})$ from $V\times V'$ into the class defined by \eqref{ints} with $Z_{s,b}$ instead of $H^s(\R)$ and $T'$ instead of $T$, is Lipschitz. 
\end{theorem}

\begin{remark}
	In case $a=0$, the idea developed below can be carried on with simpler computations and lead to a similar result. See Theorem 2.2 in \cite{HSS}. 
\end{remark}

Next we consider the IVP associated with the so-called Ostrovsky-Stepanyams-Tsimring (OST for short) equation 
\begin{equation}\label{OST}
	\begin{cases}
		\partial_t u +\partial_x^3 u -\eta(\mathcal{H}\partial_x u + \mathcal{H}\partial_x^3u)+u\partial_x u=0, \ \ \ \quad x\in\R, t>0,\\
		u(0,x)=u_0(x),
	\end{cases}
\end{equation}
where $\eta>0$ is a real constant and $\mathcal{H}$ is the Hilbert transform defined via its Fourier transform as
$$
\widehat{\mathcal{H}f}(\xi)=-i\,\mathrm{sgn}(\xi)\widehat{f}(\xi).
$$ The equation in \eqref{OST} was derived by Ostrovsky, Stepanyams and Tsimring \cite{OST} to describe the radiational instability of long waves in a stratified shear flow. The IVP \eqref{OST} in classical Sobolev spaces was considered in \cite{XM}. The authors proved the following:

\begin{innercustomthm}$($\cite[Theorem 1.1]{XM}$)$
	\label{t1.1}
	Let $u_0\in H^s(\R)$ with $s\ge0$. Then there exist $T>0$ and a unique solution of the IVP \eqref{OST} such that
	\begin{equation}\label{classost}
		\begin{array}{ll}
			u\in C([0,T];H^s(\R)), &\|\partial_x u\|_{L^2_TL^{p_1}_x}+\|D^s\partial_x u\|_{L^2_TL^{p_1}_x}<\infty, \\ \|u\|_{L^{2}_T L^{q_1}_{x}}+\|D^su\|_{L^{2}_T L^{q_1}_x}<\infty,
		\end{array}
	\end{equation}
	for $2<p_1<\infty$ and $q_1$ defined through the relation $\frac{1}{p_1}+\frac{1}{q_1}=\frac{1}{2}$. Moreover, for any $T'\in(0,T)$ there exists a neighborhood $V$ of $u_0$ in  $H^s(\R)$ such that the map $\widetilde{u}_0\mapsto \widetilde{u}(t)$, from $V$ into the class defined by \eqref{classost}, with $T'$ instead of $T$, is Lipschitz. 
\end{innercustomthm}

Using Theorem \ref{linearteo} we get the respective well-posedness in weighted spaces for $s>0$; more precisely.

\begin{theorem}\label{persistost}
	Let $u_0\in Z_{s,b}$ with $0<s<2$ and $b\le s/2$. There exist $T=T(\|u_0\|_{Z_{s,b}})>0$ and a unique $u$ in the class defined by \eqref{classost}, with $Z_{s,b}$ instead of $H^s(\R)$, which is the solution of the IVP \eqref{OST}. Moreover, for any $T'\in(0,T)$ there exists a neighborhood $V$ of $u_0$ in $Z_{s,b}$ such that the map $\widetilde{u}_0\mapsto \widetilde{u}(t)$, from $V$ into the class defined by \eqref{classost}, with $Z_{s,b}$ instead of $H^s(\R)$ and $T'$ instead of $T$, is Lipschitz. 
\end{theorem}

Next we consider two fifth-order models that  can also be handled with our theory. 
The first model we are interested in is the Kawahara equation. Consider the IVP
\begin{equation}\label{kawaivp}
	\begin{cases}
		\partial_t u + \alpha u \partial_x u+\beta\partial_x^3 u + \gamma \partial_x^5 u=0, \ \ (x,t)\in \R\times\R,\\
		u(x,0)=u_0(x),
	\end{cases}
\end{equation}
where $\alpha, \beta,\gamma$ are real numbers with $\alpha\gamma\neq0$.

The Kawahara equation was derived in \cite{Kawahara} as a model equation describing
solitary-wave propagation in media where the first-order dispersion coefficient
 is anomalously small. It also arises in modeling gravity-capillary waves on a shallow layer and magneto-sound
 propagation in plasma. Several results for the IVP  \eqref{kawaivp} can be found in the current literature. In particular, the local well-posedness in the Sobolev spaces was established in \cite{KAWA}. More precisely, the following theorem was proved.

\begin{innercustomthm}$($\cite[Theorem 3.5]{KAWA}$)$\label{localkawa}
	Let $u_0\in H^s(\R)$, $s>1/4$. There exists  $T>0$, depending on $\alpha, \beta, \gamma$ and $\|u_0\|_{s,2}$, such that  \eqref{kawaivp} has a unique solution  satisfying
	\begin{equation}\label{kawaclass}
		\begin{array}{lll}
			u\in C([-T,T];H^s(\R)),&\|u\|_{L^4_xL^\infty_T}<\infty,&\|\partial_x u\|_{L^4_TL^\infty_x}<\infty,\\ \|D^{s+2}_x u\|_{L^\infty_xL^2_T}<\infty,\ \ \ \mbox{and}
			\ \  &\|D^s\partial_x u\|_{L^4_xL^2_T}<\infty.
		\end{array}
	\end{equation}
	Moreover, for any $T'\in(0,T)$ there exists a neighborhood $V$ of $u_0$ in  $H^s(\R)$ such that the map $\widetilde{u}_0\mapsto \widetilde{u}(t)$, from $V$ into the class defined by \eqref{kawaclass}, with $T'$ instead of $T$, is Lipschitz. 
\end{innercustomthm}

By using Bourgain's spaces, the Sobolev index for the local well-posedness of \eqref{kawaivp}  may be pushed down to $s>-7/4$ (see, for instance, \cite{chenetal}). However, since Theorem \ref{localkawa} was established with the technique introduced in \cite{KPV} which uses Strichartz's estimates, smoothing effects, and a maximal function estimate combined with the contraction mapping principle, it  is enough to our purposes. Here we present  the respective local well-posedness result in weighted spaces.

\begin{theorem}\label{persistkawa}
	In addition to hypotheses of Theorem \ref{localkawa}, assume $u_0\in L^2(|x|^{2b}dx)$ for $b\le s/4$ and $1/4< s<4$. There exists $T=T(\|u_0\|_{Z_{s,b}})>0$ such that  \eqref{kawaivp} has a unique solution $u$ in the class defined by \eqref{kawaclass} with $Z_{s,b}$ instead of $H^s(\R)$. 
	Moreover, for any $T'\in(0,T)$ there exists a neighborhood $V$ of $u_0\in Z_{s,b}$ such that the map $\widetilde{u}_0\mapsto \widetilde{u}(t)$, from $V$ into the class defined by \eqref{kawaclass}, with $Z_{s,b}$ instead of $H^s(\R)$ and $T'$ instead of $T$, is Lipschitz.
\end{theorem}

The second fifth-order model we consider is when we replace the first-order derivative in the nonlinear part of \eqref{kawaivp} by a second-order derivative. More precisely,  we consider the following IVP
\begin{equation}\label{kkivp}
	\begin{cases}
		\partial_t u + \alpha u \partial_x^2 u+\beta\partial_x^3 u + \gamma \partial_x^5 u=0, \quad
		 (x,t)\in \R\times\R,\\
		u(x,0)=u_0(x),
	\end{cases}
\end{equation}
where  again  $\alpha, \beta,\gamma$ are real numbers with $\alpha\gamma\neq0$. The  local well-posedness of  \eqref{kkivp} in $H^s(\R)$  was established in \cite{KK}. More precisely, the following result was shown.

\begin{innercustomthm}$($\cite[Theorem 1.1]{KK}$)$\label{localkk}
	Suppose $\beta\gamma<0$. Let $u_0\in H^s(\R)$, $s\ge5/4$. There exists $T=T(\|u_0\|_{s,2})>0$ such that  \eqref{kkivp} has a unique solution  satisfying
	\begin{equation}\label{kklas}
		\begin{array}{llll}
			u\in C([-T,T];H^s(\R)),&\|u\|_{L^2_xL^\infty_T}<\infty,&\|\partial_x^2 u\|_{L^4_TL^\infty_x}<\infty,\ \\ \mbox{and}\ \ 	\|D^s\partial_x^2 u\|_{L^\infty_xL^2_T}<\infty.
		\end{array}
	\end{equation}
 Moreover, for any $T'\in(0,T)$ there exists a neighborhood $V$ of $u_0$ in  $H^s(\R)$ such that the map $\widetilde{u}_0\mapsto \widetilde{u}(t)$, from $V$ into the class defined by \eqref{kklas}, with $T'$ instead of $T$, is Lipschitz. 
\end{innercustomthm}

In section 4 we prove the respective well-posedness result in the weighted spaces.

\begin{theorem}\label{persistkk}
	In addition to the hypotheses of Theorem \ref{localkk}, assume $u_0\in L^2(|x|^{2b}dx)$ for $b\le s/4$ and $5/4\le s<4$. There exists $T=T(\|u_0\|_{Z_{s,b}})>0$ such that  \eqref{kkivp} has a unique solution $u$ in the class defined by \eqref{kklas} with $Z_{s,b}$ instead of $H^s(\R)$. 
Moreover, for any $T'\in(0,T)$ there exists a neighborhood $V$ of $u_0\in Z_{s,b}$ such that the map $\widetilde{u}_0\mapsto \widetilde{u}(t)$, from $V$ into the class defined by \eqref{kklas}, with $Z_{s,b}$ instead of $H^s(\R)$ and $T'$ instead of $T$, is Lipschitz.
\end{theorem}

To stand out the importance of studying dispersive equations in weighted Sobolev spaces, the final  part of this work is devoted to study dispersive blow-up properties. The phenomenon of  dispersive blow-up was first identified in \cite{BBM} for the linear KdV equation. Roughly speaking the authors proved the existence of an infinitely smooth bounded  initial data such that the corresponding solution blows-up in finite time in the $L^\infty$ norm.  The pioneer mathematical work studying the existence of solutions for nonlinear dispersive equations presenting a behavior similar to the linear KdV is due Bona and Saut \cite{Bona}. In that paper the authors considered the generalized nonlinear KdV equation 
$$
\partial_t u+\partial_x^3u+u^k\partial_xu=0, \quad k\in\Z^+,
$$
and constructed initial data in 
 $H^\ell(\R)\cap C^\infty(\R)$, for a suitable choice of $\ell$, such that the corresponding solution
 satisfies
\begin{equation}\label{disp}
	\lim_{(x,t)\to(x_*,t_*)}|u(x,t)|=+\infty,
\end{equation}
where $(x_*,t_*)$ is a point in $ \R\times (0,\infty)$; moreover, the solution $u$ is continuous except at $(x_*,t_*)$. The strategy of the authors was first to construct a solution of the linear problem satisfying \eqref{disp} and then, using the decay properties of the solutions  in weighted spaces, they showed that the nonlinear part do not destroy that behavior.   This emphasizes the linear feature of this kind of singularity and makes it different, for instance, of the blow-up in Sobolev norms where the effects of the nonlinearity are stronger. 

After that, in \cite{LinaresScialom}, the authors addressed the same question for $k\geq2$ but with a simplified approach. Indeed, the authors showed that in this situation is sufficient to show that the integral part of the solution in the Duhamel formulation is more regular than the linear one. More precisely, they established if the initial datum belong to $H^s(\R),s\geq1$, then the corresponding integral part belongs to $C(\R; H^{s+1}(\R))$. This was enough to prove the existence of dispersive blow-up. More recently, in \cite{2017}, using fractional
weighted spaces, the authors also improved the results of \cite{Bona} in the case $k = 1$, i.e., for the standard KdV equation. For additional results concerning dispersive blow-up we refer the reader to \cite{BS2}, \cite{BPSS}, \cite{Ademir}, \cite{palacios}.

 Although the ideas employed below may be applied to several models, we will pay particular attention to the Kawahara equation and the Hirota-Satsuma system. More precisely, our main theorems in this direction are listed below.

\begin{theorem}\label{dispkawa}
Assume $\gamma<0$ and $3\beta+10\gamma>0$.	There exists an initial data $u_0\in C^\infty(\R)\cap {H^{7/2}}^{-}(\R)$ such that the solution $u\in C([0,T];{H^{7/2}}^{-}(\R))$ of the IVP \eqref{kawaivp} given by Theorem \ref{persistkawa} satisfies $$u(\cdot, t^*)\in C^3(\R\setminus\{0\}) \quad \mbox{and}\quad u(\cdot, t^*)\notin C^3(\R),$$ for some $t^*\in (0,T)$.
\end{theorem}

For the Hirota-Satsuma system we  have:
\begin{theorem}\label{disphss}
	There exists an initial data $(u_0,v_0)\in \left(C^\infty(\R)\cap {H^{3/2}}^{-}(\R)\right)^2$ such that the solution $(u,v)\in \left(C([0,T];{H^{3/2}}^{-}(\R))\right)^2$ of the IVP  \eqref{system} given by Theorem \ref{persistHSS} satisfies $$(u,v)(\cdot, t^*)\in \left(C^1(\R\setminus\{0\})\right)^2\quad \mbox{and}\quad (u,v)(\cdot, t^*)\notin \left(C^1(\R)\right)^2,$$ for some $t^*\in (0,T)$.
\end{theorem}

Results similar to Theorems \ref{dispkawa} and \ref{disphss} were obtained in  \cite{2017} for the KdV, in \cite{Ademir} for the two dimensional Zakharov-Kuznetsov equation, and in \cite{palacios} for the Schr\"odinger-KdV system. We first emulate the ideas of \cite{2017} to construct a smooth initial data such that the  global solution of the associated linear IVP has an infinite number of discontinuities; at these times the linear flow cannot be smooth, which is then identified as the dispersive blow up taking place at $x=0$. Then it is shown that the Duhamel term associated to the solution of the whole IVP is smoother than the linear term of the solution, which unleash regularity on the linear term. 

This paper is organized as follows. In Section \ref{sec2} we present some notation and give some preliminary and linear results used along the paper. In particular, Theorem \ref{linearteo} is proved. Section \ref{sec3} is devoted to prove our local well-posedness results in the weighed spaces. Finally, in Section \ref{sec4} we establish the results concerning the dispersive blow-up.


\section{Preliminaries and Linear Estimates}\label{sec2}

\subsection{Notation}
Let us start by introducing some notation. We use $C$ to denote several constant that may vary from line to line. Sometimes we use subscript to indicate dependence of parameters; for instance $C_\phi$ means that the constant $C$ depends on $\phi$. We shall write $a\simeq b $, where $a$ and $b$ are two positive numbers, when there exists a constant $C>0$ such that $C^{-1}a\leq b\leq Ca$. Given a real number $r$, we use $r^+$ (respect. $r^-$) to mean $r+\varepsilon$ (respect. $r-\varepsilon$) for some sufficiently small $\varepsilon>0$.

By $L^p=L^p(\R^n)$, $1\leq p\leq +\infty$ we denote the standard Lebesgue space endowed with the usual norm. If $w$ is a weight (a nonnegative measurable function), by $L^p(wdx)$ (or $L^p(w)$ for short) we denote the space $L^p$ with respect to the measure $w(x)dx$. Given $s\in \R$, by $H^s=H^s(\R^n)$ we mean the $L^2$-based Sobolev space of order $s$. Given a function $f$ defined on $\R^n$,  $\widehat{f}$ and ${f}^\vee$ stand, respectively, for the Fourier and inverse Fourier transforms of $f$. The operators $D^s$ and $J^s$ are defined via Fourier transform as
$$
\widehat{D^sf}(\xi)=|\xi|^s\widehat{f}(\xi) \quad\mbox{and} \quad\widehat{J^sf}(\xi)=\langle\xi\rangle^s\widehat{f}(\xi),
$$
where  $\langle x\rangle:=(1+|x|^2)^{1/2}$.
For $1< p< \infty$ and $b\in\R$, the space $L^p_b(\R^n)$ is defined as $L^p_b(\R^n)=(1-\Delta)^{-b/2} L^p(\R^n)$. Note that in the case $p=2$ and $b=s$, $L^2_s(\R^n)$ is nothing but the Sobolev space $H^s(\R^n)$. In particular, the norm in $H^s(\R^n)$ is given by
$$
\|f\|_{H^s}:=\|f\|_{s,2}=\left(\int_{\R^n}(1+|\xi|^2)^s|\widehat{f}(\xi)|^2d\xi\right)^{\frac{1}{2}}.
$$
Given a function $f=f(x,t)$  of the variables $x$ and $t$, sometimes we use $\|f\|_{L^p_x}$ to indicate that we are taking the $L^p$ norm with respect to the variable $x$ only. Also, given $T>0$ we use $L^p_T$ to denote the $L^p$ space over the interval $[0,T]$.   For $1\leq q,r\leq\infty$,  the  norm in the mixed   space $L^q_{T}L^r_x$ is given by
$$
\|f\|_{L^q_{T}L^r_x}=\left\|\|f(t,\cdot)\|_{L^r_x}\right\|_{L^q_T}.
$$
Similar considerations apply to the space $L^r_xL^q_T$. In the case $r=q$ we have $\|f\|_{L^r_xL^q_T}=\|f\|_{L^q_{T}L^r_x}=\|f\|_{L^r_{xT}}$.

\subsection{Preliminaries}

In this section we discuss the technical machinery involving Stein's derivatives. Let us begin by recalling the definition of the operator $\D^b$. For any real number $b\in(0,1)$ and a measurable function $f$ define
\begin{equation*}
	\D^b(f)(x):=\displaystyle\left( \int_{\Rn}{\frac{|f(y)-f(x)|^2}{|x-y|^{n+2b}}dy} \right)^{\frac{1}{2}}.
\end{equation*}
The next theorem gives an useful characterization of the spaces $L^p_b(\R^n)$ due to Stein \cite{STEIN}.

\begin{theorem}
	\label{char}
	Let $b\in(0,1)$ and $\frac{2n}{n+2b}<p<\infty.$ A function $f$ belongs to $L^p_b(\R^n)$ if and only if  $f\in L^p(\R^n)$ and $\D^b(f)\in L^p(\Rn).$ In addition,
	\begin{equation}	\label{s,2}
		\|f\|_{b,p}:=\|f\|_{L^p_b} \simeq\|f\|_{L^p} + \|D^b(f)\|_{L^p} \simeq \|f\|_{L^p} + \|\D^b(f)\|_{L^p}.
	\end{equation}
\end{theorem}

From \eqref{s,2} one sees that the norm in $L^p_b(\R^n)$ may be given in terms of either $D^b$ or $\D^b$. The advantage of using Stein's derivative is that it is useful to perform pointwise computations.

\begin{lemma}
For	$b\in(0,1)$ and measurable functions $f$ and $g$, we have
\begin{equation}\label{2.1}
	\mathcal{D}^b(fg)(x)\le\|f\|_{L^\infty}\mathcal{D}^b(g)(x)+|g(x)|\mathcal{D}^b(f)(x)
\end{equation}
and
\begin{equation}\label{2.3}
	\|\mathcal{D}^b(fg)\|_{L^2}\le\|f\mathcal{D}^b(g)\|_{L^2}+\|g\mathcal{D}^b(f)\|_{L^2}.
\end{equation}
\end{lemma}
\begin{proof}
See Proposition 1 in \cite{NAHASPONCE}.
\end{proof}

We also may prove the following.

\begin{proposition}
	Let $b\in(0,1)$ and $p\in \Z^+$, $p\geq2$. Assume $h_i:\R^n\to\C$, $i=1,\dots,p$, are measurable. Then
	\begin{equation}\label{leibp}
		\mathcal{D}^b\left(\prod_{i=1}^{p}h_i \right)(x) \le \sum_{i=1}^{p}\mathcal{D}^b(h_i)(x)\prod_{\substack {j=1 \\ j\neq i}}^{p}\|h_j\|_{L^\infty}.
	\end{equation}
\end{proposition}
\begin{proof}
	Note that $\D^b(f)(x)$ is always a positive quantity. So the proposition follows just by iterating \eqref{2.1}.
\end{proof}

Next we establish a pointwise estimate for the Stein derivative of phase functions satisfying \eqref{AA} and \eqref{BB}.

\begin{lemma}\label{l1}
	Let $b\in(0,1)$. Suppose $\phi:\R^n\to\R$ satisfies the conditions \eqref{AA} and \eqref{BB}. For any $t\in\R$ and $x\in \R^n$ we have
	\begin{equation}\label{deriv}
		\mathcal{D}^b(e^{i t \phi(\cdot)})(x)\le C\left\{1+(1+|t|)g(x)^b \right\},
	\end{equation}
where the constant $C>0$ depends on $n,b$ and $\phi$.
\end{lemma}
\begin{proof}
We follow a similar strategy to the one applied in \cite[Proposition 2]{NAHASPONCE}.
Let $x\in \R^n$ be nonzero. Then,
	\begin{equation*}
		\mathcal{D}^b(e^{it\phi(\cdot)})(x)=\left(\int_{\R^n}\frac{\left|e^{it\phi(x)}-e^{it\phi(y)} \right|^2}{|x-y|^{n+2b}}dy  \right)^{1/2}
		=\left(\int_{\R^n}\frac{\left|e^{it(\phi(x)-\phi(y))}-1 \right|^2}{|x-y|^{n+2b}}dy  \right)^{1/2}\equiv \RNum{1}.
	\end{equation*}

To simplify notation, by $B(a,R)$ we mean the closed ball of radius $R>0$ centered at the point $a$ in $\R^n$.	Split $\R^n$ into the following three sets:
	\begin{equation*}
		\begin{array}{c}
			E_1:=B(x,g(x)^{-1})^c, \ \ \  E_2:=B(x,g(x)^{-1})\cap B(x,|x|) \ \ \ \mbox{and} \ \ \	E_3:=B(x,g(x)^{-1})\cap B(x,|x|)^c,
		\end{array}
	\end{equation*}
where $A^c$ means the complement of the set $A$ in $\R^n$.
Let $I_j$, $j=1,2,3$, be the integral $I$ with the integration over $\R^n$ replaced by the integration over $E_j$. Since, clearly, $I\leq C(I_1+I_2+I_3)$ we see that it suffices to estimate $I_j$.

In what is coming after, the inequalities
\begin{equation}\label{exp}
	|e^{i\theta}-1|\leq 2 \quad\mbox{and}\quad 	|e^{i\theta}-1|\leq |\theta|, \quad\theta\in\R,
\end{equation}
shall be used repeatedly without being mentioned.

The idea to estimate $I_j$ is to use \eqref{exp} and then to explore the radial feature of the resulting function. We begin by estimating $I_1$:
\begin{equation}\label{E1}
\begin{split}
I_1&\le \left(\int_{E_1}\frac{4}{|x-y|^{n+2b}}dy\right)^{1/2}\le C_n\left(\int_{g(x)^{-1}}^{\infty}\frac{r^{n-1}}{r^{n+2b}}dr\right)^{1/2}=C_{n,b}\left(\int_{g(x)^{-1}}^{\infty}r^{-1-2b}dr\right)^{1/2}\\ &\le
C_{n,b}\left(g(x)^{2b} \right)^{1/2}=C_{n,b}g(x)^b.
\end{split}
\end{equation}

For $I_2$ we need to divide into two cases.
	
	\underline{Case 1: $g(x)^{-1}\leq|x|$.} In this case, $E_2=B(x,g(x)^{-1})$. So, by using condition \eqref{AA} we deduce
	\begin{equation}\label{E21}
	\begin{split}
		I_2&\le C_\phi\left(\int_{E_2}\frac{|tg(x)|x-y||^2}{|x-y|^{n+2b}}dy\right)^{1/2}\leq C_{\phi}|t|g(x)\left(\int_{B(x,g(x)^{-1})}|x-y|^{2-n-2b}dw\right)^{1/2}\\&
	\le C_{\phi,n}|t|g(x)\left(\int_{0}^{g(x)^{-1}}r^{1-2b}dr\right)^{1/2}\\&
	=C_{\phi,n,b}|t|g(x)\left(g(x)^{2b-2}\right)^{1/2}= C_{\phi,n,b}|t|g(x)^b.
	\end{split}
	\end{equation}
	
	\underline{Case 2: $|x|<g(x)^{-1}$.} Here we have $E_2=B(x,|x|)\subset B(x,g(x)^{-1})$. Hence, we can use the same calculations as in Case 1 to obtain
	\begin{equation}\label{E22}
	I_2\leq \left(\int_{B(x,g(x)^{-1})}\frac{\left|e^{it(\phi(x)-\phi(y))}-1 \right|^2}{|x-y|^{n+2b}}dy  \right)^{1/2}
	\leq C_{\phi,n,b}|t|g(x)^b.
	\end{equation}
	
	Finally we estimate $I_3$. Note that $E_3$ is an annulus and it is empty if $|x|\geq g(x)^{-1}$. So we will always assume that $|x|\leq g(x)^{-1}$. Here we divide the proof into three cases.
	
	\underline{Case 1: $1\leq|x|$.} In this case we promptly obtain
	\begin{equation}\label{E31}
	\begin{split}
	I_3&\le \left(\int_{E_3}\frac{4}{|x-y|^{n+2b}}dy \right)^{1/2}=C_{n}\left(\int_{|x|}^{g(x)^{-1}}r^{n-1-n-2b}dr\right)^{1/2}\\&\le C_{n}\left(\int_{1}^{g(x)^{-1}}r^{-1-2b}dr\right)^{1/2}\le C_{n,b}\left(1-g(x)^{2b}\right)^{1/2}\le C_{n,b}. 
	\end{split}
	\end{equation}

	\underline{Case 2: $|x|<1<g(x)^{-1}$.}	We split $E_3$ into the sets 
	\begin{equation*}
		E_{31}:=E_3 \cap B(x,1) \ \ \mbox{and} \ \ E_{32}:=E_3 \cap B(x,1)^c.
	\end{equation*}
	Using condition \eqref{BB}, since $2a-1\geq1$, we get
	\begin{equation}\label{E32}
	\begin{split}
		I_3&\le \left(C_\phi\int_{E_{31}}\frac{|x-y|^{2a}}{|x-y|^{n+2b}}dy+ \int_{E_{32}}\frac{2^{2}}{|x-y|^{n+2b}}dy \right)^{1/2}\\ &\le C_{\phi,n}\left(\int_{|x|}^{1}r^{2a-1-2b}dr + \int_{1}^{g(x)^{-1}}r^{-1-2b}dr \right)^{1/2}
	\\ & \le C_{\phi,n}\left(\int_{|x|}^{1}r^{1-2b}dr + \int_{1}^{g(x)^{-1}}r^{-1-2b}dr \right)^{1/2}\\&
	\le C_{\phi,n,b}\left((1-|x|^{2-2b}) + (1-g(x)^{2b}) \right)^{1/2}\\&\le C_{\phi,n,b}.
	\end{split}
	\end{equation}

	\underline{Case 3: $g(x)^{-1}\leq1$.} Here we  use condition \eqref{BB} again to obtain
\begin{equation}\label{E33}
\begin{split}
I_3&\le C_\phi\left(\int_{E_{3}}\frac{|x-y|^{2a}}{|x-y|^{n+2b}}dy\right)^{1/2}=C_{\phi,n}\left(\int_{|x|}^{g(x)^{-1}}\frac{r^{2a+n-1}}{r^{n+2b}}dr\right)^{1/2}\\ & \le C_{\phi,n}\left(\int_{|x|}^{1}r^{2a-1-2b}dr\right)^{1/2}\le C_{\phi,n}\left(\int_{|x|}^{1}r^{1-2b}dr\right)^{1/2}\\&=C_{\phi,n,b}\left(1-|x|^{2-2b}\right)^{1/2}\le C_{\phi,n,b}. 
\end{split}
\end{equation}
From estimates \eqref{E1} to \eqref{E33} we obtain \eqref{deriv}, which proves the theorem for $x\neq 0$.

Finally, if  $x=0$ and $g(0)>0$, the proof above remains equal. In case $g(0)=0$, we divide $\R^n$ into $E_1^0=B(0,1/2)^c$ and $E_2^0=B(0,1/2)$. Note that, as  in \eqref{E1}, it can be seen that $\RNum{1}^0_1\le C_{\phi, n,b}$. Also, the argument in \eqref{E33} remains equal for $\RNum{1}^0_2$. We therefore have $\mathcal{D}^b(e^{i t \phi(\cdot)})(0) \le C_{\phi, n, b}$ and the proof of the theorem is completed.
\end{proof}

\begin{remark}
	It is worth mentioning that Lemma \ref{l1} is still valid if we impose only the weaker condition
\begin{equation}\tag{A'}
	\begin{split}
		&\text{There exists } g:\R^n\to\R \text{ measurable, } g>0 \text{ except maybe at }0,\text{ such that for all }x\in \Rn \\ &\text{ if }|x-y|\le 1\text{ then } |\phi(x)-\phi(y)|\le g(x) |x-y|. \label{AA'}
	\end{split}
\end{equation}
instead of \eqref{AA} and \eqref{BB}. The idea is to consider the sets
	\begin{equation*}
		\begin{array}{c}
			E_1:=B(x,g(x)^{-1})^c, \ \ \  E_2:=B(x,g(x)^{-1})\cap B(x,1) \ \ \ \mbox{and} \ \ \	E_3:=B(x,g(x)^{-1})\cap B(x,1)^c,
		\end{array}
	\end{equation*} and to note that the estimate of $E_3$ is the exactly \eqref{E31}.
\end{remark}

\subsection{Proof of Theorem \ref{linearteo}} This subsection is devoted to prove Theorem \ref{linearteo}. The main tool is the pointwise estimate established in Lemma \ref{l1}.

\begin{proof}[Proof of Theorem \ref{linearteo}] It suffices to prove the theorem with $s=Kb$. So, assume $f:=u_0\in L^2(|x|^{2b}dx)\cap H^{Kb}(\R^n)$. We already now that $L$ generates a unitary group, say,  $\{U(t)\}$ in $H^{Kb}(\R^n)$ such that $U(t)f=(e^{-it\Phi(\cdot)}\widehat{f})^\vee$. From Plancherel's theorem  and \eqref{s,2} we have
	\begin{equation}\label{estrela2*}
\begin{split}
	\||x|^bU(t)f\|_{L^2}&=\|D^b(e^{-it\Phi(\cdot)}\widehat{f})\|_{L^2}\\&
\le C \|e^{-it\Phi(\cdot)}\widehat{f}\|_{L^2}+C\|\mathcal{D}^b(e^{-it\Phi(\cdot)}\widehat{f})\|_{L^2}\\&\le 
C\|f\|_{L^2}+\|\mathcal{D}^b(e^{-it\Phi(\cdot)}\widehat{f})\|_{L^2}.
\end{split}
	\end{equation}
Hence we need to estimate the quantity $\|\mathcal{D}^b(e^{-it\Phi(\cdot)}\widehat{f})\|_{L^2}$.
According to \eqref{2.3} and \eqref{leibp} we have
\begin{equation}\label{estrela2}
	\begin{split}
	 \|\mathcal{D}^b(e^{-it\Phi(\cdot)}\widehat{f})\|_{L^2}&\le \|\widehat{f}\mathcal{D}^b(e^{-it\Phi(\cdot)})\|_{L^2}+\|e^{-it\Phi(\cdot)}\mathcal{D}^b(\widehat{f})\|_{L^2} \\&\le \left\|\widehat{f}\mathcal{D}^b\left(\prod_{i=1}^p e^{-it\phi_i(\cdot)}\right)\right\|_{L^2}+\|\mathcal{D}^b(\widehat{f})\|_{L^2}\\&\le
		\left\|\widehat{f}\sum_{i=1}^{p}\mathcal{D}^b(e^{-it\phi_i(\cdot)})\cdot 1\right\|_{L^2}+\|\mathcal{D}^b(\widehat{f})\|_{L^2}.
	\end{split}
\end{equation}
In view of Lemma \ref{l1},
\begin{equation}\label{estrela3}
\begin{split}
	 \left\|\widehat{f}\sum_{i=1}^{p}\mathcal{D}^b(e^{-it\phi_i(\cdot)})\right\|_{L^2}&\le C \left\|\widehat{f}\sum_{i=1}^{p}\left\{1+(1+|t|)g_i(x)^{b}  \right\}\right\|_{L^2} \\
	&\le C\left\|\widehat{f}\sum_{i=1}^{p}\left\{1+(1+|t|)(1+|x|^{bk_i})  \right\}\right\|_{L^2} \\
	&\le C\left\|\widehat{f}\sum_{i=1}^{p}\left\{1+(1+|t|)(2+|x|)^{bK}  \right\}\right\|_{L^2}
	\\ &\le C(1+|t|)\left\|(1+|x|)^{bK}\widehat{f}\right\|_{L^2} \\ 
	 & \le  C(1+|t|)\|f\|_{bK,2},
\end{split}
\end{equation}
where the constant $C$ depends on $n,b,K$ and $p$. Moreover, since $f\in L^2(|x|^{2b}dx)\cap L^2(\R^n)$ we have  $\widehat{f}\in H^b(\R^n)$ and by Theorem \ref{char},
\begin{equation} 
		\|\mathcal{D}^b(\widehat{f})\|_{L^2}\le C \|\widehat{f}\|_{L^2} + C\|D^b(\widehat{f})\|_{L^2}=C\|f\|_{L^2}+C\||x|^bf\|_{L^2}. \label{estrela4}
	\end{equation}
Gathering together estimates \eqref{estrela2*}-\eqref{estrela4} the proof of Theorem \ref{linearteo} is complete. 
\end{proof}

\subsection{Commutator and interpolation estimates}
We end this section by recalling some commutator and interpolation estimates which will be useful below. We start with the following commutator estimate for homogeneous derivatives.

\begin{lemma}\label{A12}
	Let $s\in(0,1)$. Then
\begin{itemize}
	\item[(i)]  For $1<p<\infty$,
	\begin{equation*}
		\|D^s(fg)-fD^sg-gD^sf \|_{L^p}\le C \|g\|_{L^\infty} \|D^sf\|_{L^p}.
	\end{equation*}
\item[(ii)] For $1<r,p_1,p_2,q_1,q_2<\infty$ satisfying
$$
\frac{1}{r}=\frac{1}{p_1}+\frac{1}{q_1}=\frac{1}{p_2}+\frac{1}{q_2}
$$
it holds
$$
\|D^s(fg)\|_{L^r}\leq C\|f\|_{L^{p_1}}\|D^sg\|_{L^{q_1}}+C\|D^sf\|_{L^{p_2}}\|g\|_{L^{q_2}}.
$$

\item[(iii)] For  $1<p_1,p_2,q_1,q_2<\infty$ satisfying
$$
1=\frac{1}{p_1}+\frac{1}{p_2}, \qquad \frac{1}{2}=\frac{1}{q_1}+\frac{1}{q_2}
$$

we have
\begin{equation*}
	\|D^s(fg)-fD^sg-gD^sf \|_{L^1_xL^2_T}\le C \|g\|_{L^{p_1}_xL^{q_1}_T} \|D^sf\|_{L^{p_2}_xL^{q_2}_T}.
\end{equation*}\end{itemize}
\end{lemma}
\begin{proof}
For part (i) see Theorem A.12 in \cite{KPV}. For part (ii) see Proposition 3.3 in \cite{book}. For (iii) see Theorem A.13 in \cite{KPV}.
\end{proof}

In the next lemma $\mathcal{S}(\R^n)$ stands for the Schwartz space and $A_p$ denotes the Muckenhoupt class on $\R^n$. More precisely, given $1<p<\infty$, the Muckenhoupt
class $A_p$ consists of all weights $\omega$ such that
\begin{equation}
	[\omega]_p=\sup_{Q}\left(\frac{1}{|Q|}\int_{Q}\omega(y)dy \right)\left(\frac{1}{|Q|}\int_{Q}\omega^{-\frac{1}{p-1}}(y)dy \right)^{p-1}<\infty,
\end{equation}
where the supremum is taken over all cubes $Q\subset\R^n$; additional details and properties may be seen in \cite{commuter} and \cite{Hunt}.

The next result is a version of the Kato-Ponce commutator estimate in weighted spaces.

\begin{lemma}\label{commuter}
	Let $1<p,q<\infty$ and $1/2<\ell<\infty$ be such that $\frac{1}{\ell}=\frac{1}{p}+\frac{1}{q}$.  If $v\in A_p$, $w\in A_q$ and $s>\max\{0, n(\frac{1}{\ell}-1)\}$ or $s$ is a non negative even integer, then for all $f,g\in\mathcal{S}(\R^n)$ we have
	\begin{equation}\label{commuterb}
		\|D^s(fg)-fD^sg\|_{L^\ell(v^{\frac{\ell}{p}} w^{\frac{\ell}{q}})} \le C \|D^s f\|_{L^p(v)}\|g\|_{L^q(w)}+\|\nabla f\|_{L^p(v)}\|D^{s-1}g\|_{L^q(w)}.
	\end{equation}
\end{lemma}
\begin{proof}
	See Theorem 1.1 in \cite{commuter}.
\end{proof}

We also need the following characterization for the boundedness of the Hilbert transform in weighted spaces.

\begin{lemma}\label{hilb}
The Hilbert transform is bounded in $L^p(wdx)$, $1<p<\infty$, if and only if $w\in A_p$.
\end{lemma}
\begin{proof}
	See Theorem 9 in \cite{Hunt}.
\end{proof}

We finally introduce two interpolation inequalities. 

\begin{lemma}
	Assume $a, b>0$, $1<p<\infty$ and $\theta\in(0,1)$. If $J^a f\in L^p(\R^n)$ and $\langle x \rangle^b f\in L^p(\R^n)$ then 
	\begin{equation}\label{interpp}
		\|\langle x\rangle^{(1-\theta)b}J^{\theta a}f\|_{L^p(\R^n)}\le C \|\langle x\rangle^b f\|_{L^p(\R^n)}^{1-\theta}\|J^a f\|_{L^p(\R^n)}^\theta.
	\end{equation}
	The same holds for $D$ instead of $J$. Moreover, for $p=2$ we have
	\begin{equation}\label{interp2}
		\left\|J^{\theta a}\left(\langle x\rangle^{(1-\theta)b}f\right)\right\|_{L^2(\R^n)}\le C \|\langle x\rangle^b f\|_{L^2(\R^n)}^{1-\theta}\|J^a f\|_{L^2(\R^n)}^\theta.
	\end{equation}
\end{lemma}
\begin{proof}
Inequality	\eqref{interp2} follows from \eqref{interpp} in view of Plancherel's identity. For the proof of \eqref{interpp} see Lemma 4 in \cite{NAHASPONCE} and Lemma 2.7 in \cite{Ademir}.
\end{proof}

\section{Local well-posedness in weighted spaces} \label{sec3}

This section is devoted to prove our local well-posedness results in the spaces $Z_{s,b}$. In all cases, the main idea is to use the technique introduced in \cite{KPV} which combines Strichartz-type estimates, Kato's smoothing effects and a maximal function estimate with the contraction mapping principle to obtain a unique fixed point (the solution) of the corresponding integral equation.

\subsection{The Hirota-Satsuma system} \label{sec3.1}
Denote by $U_a(t)$ the unitary group associated with the linear part of the first equation in \eqref{system}, that is, $U_a(t)f=(e^{-ita\xi^3}\widehat{f})^\vee$ and set $U(t)\equiv U_{-1}(t)$. It is clear that conditions (A) and (B) are satisfied by the phase function $\Phi(x)=ax^3$, for any $a\neq0$.

Before proving Theorem \ref{persistHSS}, we recall the strategy to prove Theorem \ref{borys}. For $T>0$  set 
\[
\begin{split}
		\Lambda_s^T(w)&:=\max_{[-T,T]}\|w(t)\|_{s,2} + \|\partial_x w\|_{L^4_T L^\infty_x}+\|D^s_x\partial_x w\|_{L^\infty_x L^2_T} + (1+T)^{-1/2}\|w\|_{L^2_xL^\infty_T}+\|\partial_x w\|_{L^\infty_x L^2_T}.
\end{split}
\]
In \cite{HSS} it was shown that the map $\Psi(u,v)=(\Psi_1(u,v),\Psi_2(u,v))$ defined by
\begin{equation*}
	\left\{	\begin{array}{l}
		\Psi_1(u,v)(t)=U_a(t)u_0+\displaystyle\int_{0}^{t}U_a(t-t^\prime)(6au\partial_xu - 2rv\partial_x v)(t^\prime)dt^\prime, \\
		\Psi_2(u,v)(t)=U(t)v_0-3\displaystyle\int_{0}^{t}U(t-t^\prime)(u\partial_x v)(t^\prime)dt^\prime,
	\end{array}\right.
\end{equation*}
is a contraction in the space 
$$X^T_M := \{(u,v)\in C([-T,T],H^s(\R))\times 
C([-T,T],H^s(\R)) \mid \Lambda_s^T(u)+\Lambda_s^T(v)\le M \}, 
$$
for a suitable choice of the parameters $T$ and $M$ with
\begin{equation}\label{hiro1}
\Lambda_s^T(\Psi_1(u,v))+\Lambda_s^T(\Psi_2(u,v))\le C\|u_0\|_{s,2}+C\|v_0\|_{s,2}+CT^{1/2}(T^{1/4}+(1+T)^{1/2})M^2,
\end{equation}
for some universal constant $C>0$ and any $(u,v)\in X_M^T$. From the contraction mapping principle one obtains the unique solution.

\begin{proof}[Proof of Theorem \ref{persistHSS}]	
We follow the same strategy described above.	Consider $$\lambda^T(w):=\max\limits_{[-T,T]}\||x|^b w\|_{L^2_x}.$$  We are going to prove that $\Psi(u,v)$ is a contraction in the space 
	$$Y^T_M := \{(u,v)\in C([-T,T],Z_{s,b})\times 
	C([-T,T],Z_{s,b}) \mid \Omega_s^T(u)+\Omega_s^T(v)\le M \}, $$ endowed with the norm $\vertiii{(u,v)}:=\Omega_s^T(u)+\Omega_s^T(v)$, where $\Omega_s^T(w)=\Lambda_s^T(w)+\lambda^T(w)$ and $T,M>0$ will be determined later. 
	
	We begin by estimating $\Psi_1(u,v)$ for $(u,v)\in Y_M^T$. In view of \eqref{hiro1} it suffices to estimate $\lambda^T(\Psi_1(u,v))$. Using Minkowski's inequality we obtain
	\begin{equation}
	\begin{split}
	\||x|^b\Psi_1(u,v)\|_{L^2_x}&\le \||x|^bU_a(t)u_0\|_{L^2_x}+\displaystyle\int_{0}^{T}\||x|^bU_a(t-t')(6au\partial_xu+2rv\partial_x v)(t')\|_{L^2_x}dt' \\ 
		&\le \||x|^bU_a(t)u_0\|_{L^2_x}+\displaystyle\int_{0}^{T}\|(|x|^bU_a(t-t')6au\partial_xu)(t')\|_{L^2_x}dt' \\
		 &\quad +\displaystyle\int_{0}^{T}\|(|x|^bU_a(t-t')2rv\partial_x v)(t')\|_{L^2_x}dt'\\
		 & \le \RNum{1}+\RNum{2}+\RNum{3}.\label{tele1}
	\end{split}
	\end{equation}
In view of \eqref{trocas},
	\begin{align}\label{doo1}
		\RNum{1}&\le C\||x|^bu_0\|_{L^2_x} + C(1+T)\|u_0\|_{s,2}. 
	\end{align}
for some positive constant $C$ (depending on $s$). Another application of \eqref{trocas} combined with H\"older's inequality gives
\[
\begin{split}
		\RNum{2}&\le \displaystyle\int_{0}^{T}C\|(|x|^b u\partial_xu)(t')\|_{L^2_x}+C(1+T)\|(u\partial_xu)(t')\|_{s,2}dt'\\
	&\le CT^{1/2}(1+T)\left\{ \||x|^b u\partial_x u \|_{L^2_TL^2_x}+\|u\partial_x u \|_{L^2_TL^2_x} + \|D^s_x(u\partial_x u) \|_{L^2_TL^2_x}\right\}
\end{split}
\]
Since $\Lambda_s^T$ contains the $L^\infty_TH^s$ norm, the last two terms in the above inequality have already been estimated in \cite[Theorem 2.1]{HSS}; more precisely,
\begin{equation}\label{tele0}
	\|u\partial_x u \|_{L^2_TL^2_x} + \|D^s_x(u\partial_x u) \|_{L^2_TL^2_x}\leq  CT^{1/2}(T^{1/4}+(1+T)^{1/2})M^2.
\end{equation}
To bound the remaining term we use H\"older's inequality to deduce
\[
\begin{split}
 \||x|^b u\partial_x u \|_{L^2_TL^2_x}\leq 	T^{1/4} \displaystyle\max_{[-T,T]}\||x|^bu\|_{L^2_x}\|\partial_xu\|_{L^4_TL^\infty_x} \leq T^{1/4}(\Omega_s^T(u))^2\leq T^{1/4}M^2.
\end{split}
\] 
Hence
	\begin{equation}\label{localt}
		\begin{array}{l}
	\RNum{2}\le
	 CT^{1/2}(1+T)(T^{1/4}+(1+T)^{1/2})M^2.
		\end{array}
	\end{equation}
A similar computation establishes
	\begin{equation}
		\RNum{3}\le CT^{1/2}(1+T)(T^{1/4}+(1+T)^{1/2})M^2.\label{doo3}
	\end{equation}
	
	Estimates \eqref{doo1}-\eqref{doo3} yield
	\begin{equation*}
		\lambda^T(\Psi_1(u,v))\le C\||x|^bu_0\|_{L^2_x} + C(1+T)\|u_0\|_{s,2} + CT^{1/2}(1+T)(T^{1/4}+(1+T)^{1/2})M^2.
	\end{equation*}
	
By using the same argument it can be seen that
	\begin{align*}
		\lambda^T(\Psi_2(u,v))&\le \||x|^bU(t)v_0\|_{L^2_x}+\displaystyle\int_{0}^{T}\||x|^bU(t-t')(u\partial_x v)(t')\|_{L^2_x}dt' \\ &\le C\||x|^bv_0\|_{L^2_x} + C(1+T)\|v_0\|_{s,2} + CT^{1/2}(1+T)(T^{1/4}+(1+T)^{1/2})M^2.
	\end{align*}

Collecting these estimates we get
\[
\begin{split}
\Omega_s^T(\Psi_1(u,v))+\Omega_s^T(\Psi_2(u,v))&\leq C\Big\{\||x|^bu_0\|_{L^2_x}+|x|^bv_0\|_{L^2_x}+(1+T)(\|u_0\|_{s,2}+\|v_0\|_{s,2})\Big\}\\
&\quad +CT^{1/2}(1+T)(T^{1/4}+(1+T)^{1/2})M^2.
\end{split}
\]
By choosing $$M=2C\Big\{\||x|^bu_0\|_{L^2}+\||x|^bv_0\|_{L^2}+2(\|u_0\|_{s,2}+\|v_0\|_{s,2}) \Big\}$$ and $0<T\leq1$ sufficiently small such that $$2CT^{1/2}(1+T)(T^{1/4}+(1+T)^{1/2})M\le 1$$ we deduce  that $\Psi:Y^T_M\to Y^T_M$ is well defined.  Moreover, similar arguments
show that $\Psi$ is a contraction. The rest of the proof follows from standard arguments; thus
we omit the details. 
\end{proof} 

\subsection{The OST equation}

In \cite{XM}, to prove Theorem \ref{t1.1}, besides Strichartz's estimates, the authors used the contraction principle with a refined smoothing effect for the semigroup \begin{equation}\label{jazz}V(t)u_0=\left(e^{-it\Phi(\cdot)}\widehat{u_0} \right)^\vee, \ \ \mbox{where} \ \Phi(\xi)=-\xi^3-\eta(|\xi|-|\xi|^3).\end{equation}
In particular the next lemma was established.

\begin{lemma}\label{corol}
	If $u_0\in H^s(\R)$, $0<s\le 1$, $0<T<1$ and $\gamma:=\min\{\frac{1}{2}, \frac{2s}{3}\}$, then 
	\begin{equation*}
		\|\partial_x V(t)u_0\|_{L^2_T L^\infty_x}\le CT^\gamma \|D^s_xu_0\|_{L^2},
	\end{equation*}
for some constant $C>0$ depending on $\eta$ and $s$.
\end{lemma}
\begin{proof}
 See Corollary 2.2 in \cite{XM}.	
\end{proof}

As before, we rush an overview of the proof of Theorem \ref{t1.1}. Consider the space $$X_M^T=\{w\in C([0,T];H^s(\R)) \mid \Lambda^T(w)\le M \},$$ with
\begin{align*}
	\Lambda^T(w)=\sum_{i=1}^{5}\lambda_i^T(w):=&\max_{[0,T]}\|w\|_{s,2}+ \|\partial_x w\|_{L^2_TL^{p_1}_x}+\|D^s\partial_x w\|_{L^2_T L^{p_1}_x}\\&+ T^{-\gamma(p_1)}\|w\|_{L^2_TL^{q_1}_x}+ T^{-\gamma(p_1)}\|D^s w\|_{L^2_TL^{q_1}_x},
\end{align*}
where $\gamma(p_1)$ is a positive constant depending only on $p_1$.
The authors, in \cite{XM} then proved that the map $\Psi:X_M^T\to X^T_M$, defined by
\begin{equation*}
	\Psi(u)(t)=V(t)u_0-\int_{0}^{t}(V(t-t')u\partial_x u )(t')dt',
\end{equation*} 
is a contraction, for a suitable choice of $T$ and $M$ satisfying
\begin{equation}\label{chihcarra0}
	\Lambda^T(\Psi(u))\le C\|u_0\|_{s,2}+CT^{\gamma(p_1)}M^2,
\end{equation}
for some positive constant $C$ and any $u\in X_M^T$.

In order to prove Theorem \ref{persistost}, note that the phase $\Phi$ in \eqref{jazz} satisfy the conditions of Theorem \ref{linearteo} because it is a combination of particular cases of functions $\phi$ mentioned in the introduction. 
We therefore may use \eqref{trocas}. 

\begin{proof}[Proof of Theorem \ref{persistost}] We provide details for the computations when $0<s<1$. 
	Set $\gamma=\min\{1/2, 2s/3\}.$ For $0<T<1$, in addition to the norms in $\Lambda^T$, consider $\lambda_6^T(w):=T^{-\gamma}\|\partial_x w\|_{L^2_TL^\infty_x}$ and $\lambda_7^T(w):=\||x|^b w\|_{L^\infty_TL^2_x}$.
	Define $$Y^T_M:=\{ w\in C([0,T];Z_{s,b}) \mid \Omega^T(w)\le M  \}\ \ \mbox{where} \ \ \Omega^T(w)=\Lambda^T(w)+\lambda_6^T(w)+\lambda_7^T(w).$$ 

We will show that for suitable choices of $M$ and $T$, the map $\Psi:Y_M^T\to Y_M^T$ is well defined and is a contraction. From \eqref{chihcarra0} it remains to estimate the norms $\lambda_6^T$ and $\lambda_7^T$. In view of Lemma \ref{corol} we have	\begin{equation}\label{chicharra}
	\begin{split}
	 \lambda_6^T(\Psi(u))&\le T^{-\gamma}\|\partial_xV(t)u_0\|_{L^2_TL^\infty_x}+T^{-\gamma}\left\|\partial_xV(t)\int_{0}^{t}(V(-t')u\partial_x u)(t')dt'\right\|_{L^2_TL^\infty_x}\\
		&\le C\|D^su_0\|_{L^2_x}+C\left\| D^s_x\int_{0}^{t}(V(-t')u\partial_x u)(t')dt'\right\|_{L^2_x}\\
		&\leq C\|u_0\|_{s,2}+\int_{0}^{T}\|D^s_x(u\partial_x u)(t')\|_{L^2_x}dt'\\
		&\equiv C(\|u_0\|_{s,2}+\RNum{1}).
	\end{split}
	\end{equation}
According to the fractional Leibniz rule (see Lemma \ref{A12}) we have
	\begin{align*}
		\|D^s_x(u\partial_x u)\|_{L^2}&\le  C\|u\|_{L^{q_1}}\|D^s_x\partial_x u\|_{L^{p_1}}+C\|\partial_xu\|_{L^{p_1}}\|D^s_x u\|_{L^{q_1}}.
	\end{align*}
	Therefore, from H\"older's inequality, we deduce
	$$I\le C\|u\|_{L^2_TL^{q_1}_x}\|D^s_x\partial_x u\|_{L^2_T L^{p_1}_x}+ C\|\partial_x u\|_{L^2_TL^{p_1}_x}\|D^s_x u\|_{L^2_TL^{q_1}_x}\le CT^{\gamma(p_1)}\Omega^T(u)^2.$$
	We conclude from \eqref{chicharra} that 
	$$\lambda_6^T(\Psi(u))\le C\|u_0\|_{s,2}+CT^{\gamma(p_1)}\Omega^T(u)^2.$$
	
Besides, using Theorem \ref{linearteo}, we get
\begin{equation}
	\begin{split}
	\||x|^b\Psi(u)\|_{L^2_x}&\le \||x|^bV(t)u_0\|_{L^2_x}+\int_{0}^{T}\left\||x|^b V(t-t')(u\partial_x u)(t') \right\|_{L^2}dt'\\
	&\le
		C(1+T)\|u_0\|_{s,2}+C\||x|^b u_0\|_{L^2_x}+C\int_{0}^T(1+T)\|u\partial_x u\|_{s,2}dt'\\
		&\quad+C\int_{0}^{T}\||x|^bu\partial_xu\|_{L^2_x}dt'\\
		&=C(1+T)\|u_0\|_{s,2}+C_s\||x|^b u_0\|_{L^2_x}+\RNum{2}+\RNum{3}.\label{022}
	\end{split}
\end{equation}
The term $\RNum{2}$ can be estimated as done with $\RNum{1}$ (actually, this term has already been estimated in the $H^s(\R)$ local theory). In particular, we obtain
	\begin{equation}
		\RNum{2}\le C(1+T)T^{\gamma(p_1)}\Lambda^T(u)^2\le C(1+T)T^{\gamma(p_1)}\Omega^T(u)^2.
	\end{equation}
	In what comes to $\RNum{3}$ we use H\"older's inequality as follows:
	\begin{align}
		\RNum{3}&\le CT^{1/2}\||x|^bu\partial_x u\|_{L^2_TL^2_x}\le CT^{1/2}\max_{[0,T]}\||x|^bu\|_{L^2_x}\|\partial_xu\|_{L^2_TL^\infty_x}	\nonumber\\&\le CT^{1/2+\gamma}\lambda_6^T(u)\lambda_7^T(u)\le CT^{1/2+\gamma}\Omega^T(u)^2.\label{024}
	\end{align}
	From \eqref{022}-\eqref{024}	we conclude
	\begin{equation*}
		\lambda_7^T(u)\le C(1+T)\|u_0\|_{s,2}+C\||x|^b u_0\|_{L^2_x}+C(1+T)(T^{1/2+\gamma}+T^{\gamma(p_1)})\Omega^T(u)^2.
	\end{equation*}

Gathering together the above estimates we finally obtain
$$
\Omega^T(\Psi(u))\le C(1+T)\|u_0\|_{s,2}+C\||x|^b u_0\|_{L^2_x}+C(1+T)(T^{1/2+\gamma}+T^{\gamma(p_1)})\Omega^T(u)^2
$$
	
	By setting $M=2C\left\{2\|u_0\|_{s,2}+\||x|^b u_0\|_{L^2}\right\}$ and taking $0<T<1$ such that $$C(1+T)(T^{1/2+\gamma}+T^{\gamma(p_1)})M\le \frac{1}{2}$$ it can be seen that $\Phi:Y^T_M\to Y^T_M$ is well defined. Moreover, similar arguments show that
	$\Psi$ is a contraction. To finish the proof we use standard arguments, thus, we omit the details.
\end{proof}

\begin{remark}
	In \cite{esfahani}, using a purely dissipative method, the author established the local well-posedness of \eqref{OST} in $H^s(\R)$ for $s>-3/2$. However, as we already said,  the relation between decay and low regularity is not well understood; so, we are not able to establish a local well-posedness result in $Z_{s,b}$ for indices $s\leq 3/4$.
\end{remark}

\subsection{Kawahara equation}

Denote by $W(t)$ the unitary group associated to the linear part of the problem \eqref{kawaivp}, that is, \begin{equation}\label{necesidade}W(t)u_0(x)=\left(e^{it(-\gamma\xi^5+\beta\xi^3)}\widehat{u_0}\right)^\vee(x).
\end{equation}
For $M,T>0$ and $s>1/4$, consider the space
$$X_M^T:=\{w\in C([-T,T];H^s(\R))\mid \Lambda^T(w)\le M\},$$
where
$$
\Lambda^T(w):=\max_{[-T,T]}\|w\|_{s,2}+\|\partial_x w\|_{L^4_TL^\infty_x}+\|w\|_{L^4_xL^\infty_T} + \|D^{s+2}w\|_{L^\infty_xL^2_T}+\|D^{s}_x\partial_xw\|_{L^4_xL^2_T}
$$
In \cite{KAWA} the authors showed that the  integral equation
\begin{equation*}
	\Psi(u)(t)=W(t)u_0+\alpha\int_{0}^{t}W(t-t')(u\partial_x u)(t')dt'
\end{equation*} 
is a contraction in $X_M^T$  with
\begin{equation}\label{messi0}
	\Lambda^T(\Psi(u))\leq C\|u_0\|_{s,2}+CT^{1/2}\Lambda^T(u)^2,
\end{equation}
for some $C>0$ and any $u\in X_M^T$.

Moreover, the following lemma was established:
\begin{lemma}\label{lemonkawa}
	Let $s>1/4$ and $0<T\leq1$. If $\Lambda^T(u)<\infty$ then $u\partial_x u \in L^2([-T,T];H^s(\R))$ and $$\left(\int_{-T}^T\left\|(u\partial_x u)(t')\right\|_{s,2}^2dt'\right)^{1/2}\le C\Lambda^T(u)^2.$$ 
\end{lemma}
\begin{proof}
	 See Lemma 3.3 in \cite{KAWA}.
\end{proof}

Note that the phase function $\Phi(x)=-\gamma x^5+\beta x^3$  is in the scope of Theorem \ref{linearteo}. Hence, we are in a position to prove Theorem \ref{persistkawa}.

\begin{proof}[Proof of Theorem \ref{persistkawa}]
	Set $\lambda_6^T(w):=\max\limits_{[-T,T]}\||x|^bw\|_{L^2_x}$ and consider the space $$Y_M^T:=\{w\in C([-T,T];Z_{s,b}) \mid \Omega^T(w)\le M \},\ \ \mbox{where}\ \ \Omega^T(w)=\Lambda^T(w)+\lambda_6^T(w).$$
	To see that $\Psi$ maps $Y_M^T$ into itself we need to estimate it in the norm $\lambda_6^T$. For any $u\in X_M^T$, using \eqref{trocas} and H\"older's inequality we get
	\begin{align}\label{messi1}
		\||x|^b\Psi(u)\|_{L^2_x}&\le C\left\{(1+T)\|u_0\|_{s,2}+\||x|^bu_0\|_{L^2_x}+T^{1/2}(1+T)\|u\partial_x u\|_{L^2_TH^s_x}\nonumber\right.\\&\left.\hspace{63mm}+T^{1/2}\||x|^b u \partial_x u\|_{L^2_TL^2_x}  \right\}.
	\end{align}
	According to Lemma \ref{lemonkawa} we have $\|u\partial_x^2 u\|_{L^2_TH^s_x}\le C\Lambda^T(u)^2$. Besides, using H\"older's inequality we obtain
	\begin{equation}\label{messi2}
		\||x|^bu\partial_x u\|_{L^2_TL^2_x}\le \max_{[-T,T]}\||x|^bu\|_{L^2_x}\|\partial_xu\|_{L^2_TL^\infty_x}\le T^{1/4}\lambda_6^T(u)\|\partial_x u\|_{L^4_TL^\infty_x}\leq T^{1/4}\Omega^T(u)^2.
	\end{equation}
	Hence, from \eqref{messi1} and \eqref{messi2} we conclude
	\begin{equation}\label{messi3}
		\lambda_6^T(u)\le C\left\{(1+T)\|u_0\|_{s,2}+\||x|^bu_0\|_{L^2_x}+(1+T)(T^{3/4}+T^{1/2})\Omega^T(u)^2  \right\}.
	\end{equation}
	Finally, by combining \eqref{messi0} and \eqref{messi3} we obtain
	$$
	\Omega^T(\Psi(u))\le C\left\{(1+T)\|u_0\|_{s,2}+\||x|^bu_0\|_{L^2_x}+(1+T)(T^{3/4}+T^{1/2})\Omega^T(u)^2 \right\}.
	$$
	By taking $M=2C\left\{2\|u_0\|_{s,2}+\||x|^bu_0\|_{L^2_x}\right\}$ and $0<T<1$ such that $$C(1+T)(T^{3/4}+T^{1/2})M<\frac{1}{2},$$ we infer that $\Psi:Y_M^T\to Y^T_M$ is well defined. The rest of the proof runs from standard arguments.
\end{proof}

\subsection{The fifth-order equation}
Our goal here is to prove Theorem \ref{persistkk}.
For positive constants $T$ and $M$ consider the space $$X_M^T:=\{w\in C([-T,T];H^s(\R))\mid \Lambda^T(w)\le M\},$$
where
$$
\Lambda^T(w):=\max_{[-T,T]}\|w\|_{s,2}+\|\partial_x^2 w\|_{L^4_TL^\infty_x}+\|w\|_{L^2_xL^\infty_T}+  \|D^{s+2}w\|_{L^\infty_xL^2_T}.
$$
Let $W(t)$ be as in \eqref{necesidade}. For  $s\ge5/4$, and suitable choices of $T$ and $M$,  in \cite{KK} the authors showed that the integral equation
\begin{equation*}
	\Psi(u)(t)=W(t)u_0+\alpha\int_{0}^{t}W(t-t')(u\partial_x^2 u)(t')dt'
\end{equation*} 
maps $X_M^T$ into itself, is a contraction and satisfies
\begin{equation*}
	\Lambda^T(\Psi(u))\leq C\|u_0\|_{s,2}+CT^{1/2}\Lambda^T(u)^2,
\end{equation*}
for some $C>0$ and any $u\in X_M^T$.

Moreover, they showed the following lemma:
\begin{lemma}\label{lemon}
	Let $0\le T<1$ and $s\ge5/4$. If $\Lambda^T(u)<\infty$ then $u\partial_x^2 u \in L^2([-T,T];H^s(\R))$ and $$\left(\int_{-T}^T\left\|(u\partial_x^2 u)(t')\right\|_{s,2}^2dt'\right)^{1/2}\le C\Lambda^T(u)^2,$$ where $C>0$ depends only on $\alpha, \beta$,  $\gamma$, and s.
\end{lemma}
\begin{proof}
	See Lemma 3.2 in \cite{KK}.
\end{proof}

\begin{proof}[Proof of Theorem \ref{persistkk}]
	The proof follows by setting $\lambda_5^T(w):=\max\limits_{[-T,T]}\||x|^bw\|_{L^2_x}$ and arguing as in the proof of Theorem \ref{persistkawa}.
\end{proof}

\section{Dispersive blow up} \label{sec4}

In this section we use the local theory developed above to study dispersive blow up properties regarding the Kawahara equation and the Hirota-Satsuma system. 

\subsection{The Kawahara Equation}\label{KEsec}
We prove Theorem \ref{dispkawa} in the following two steps. We first build an initial data satisfying the conditions listed in Proposition \ref{step1kawa} (bellow) and then we prove a nonlinear smoothing effect that reduces the regularity properties of the solution to the linear term.  

\subsubsection{Construction of the initial data}
Let $f:\R\to \R$ be defined by $f(x):=e^{-2|x|}$. Set $\phi(x):=(f*f)(x)=\frac{1}{2}e^{-2|x|}(1+2|x|)$. It is not difficult to see that $\phi\in {H^{7/2}}^{-}(\R)\cap L^2({\langle x\rangle^{7/4}}^-dx)$, $\phi\in C^3(\R\setminus\{0\})\setminus C^3(\R)$,  $e^x\phi\in L^2(\R)$, and $e^{-x}\phi\in L^2(\R)$.

Assume for the moment that $u_0$ has the form
\begin{equation}\label{dadokawa}
	u_0(x):=\sum_{j=1}^{\infty}\alpha_j W(-\sigma j)\phi(x),
\end{equation}
where $W(t)$ is the unitary group defined in \eqref{necesidade}, $\sigma>0$ is fixed and $\alpha_j$ will be defined later. 

\begin{proposition}\label{step1kawa}
	Assume $\gamma<0$ and $3\beta+10\gamma>0$. For any $\sigma>0$ there exists a sequence $\{\alpha_j\}$ such that the function $u_0$ in \eqref{dadokawa} belongs to $C^\infty(\R)\cap{H^{7/2}}^-(\R) \cap L^2({\langle x\rangle^{7/4}}^-dx)$. In addition, the associated global in-time solution $u\in C(\R;{H^{7/2}}^-(\R))$ of the linear part of the IVP \eqref{kawaivp} satisfies
	\begin{itemize}
		\item[(i)] For any $t>0$ with $t\notin \sigma\Z^+$ we have $u(\cdot,t)\in C^\infty(\R).$
		\item[(ii)] For any $t\in\sigma\Z^+$ we have $u(\cdot,t)\in C^3(\R\setminus\{0\})\setminus C^3(\R)$. 
	\end{itemize}
\end{proposition}
\begin{proof}
	The proof is based on Section 3 of \cite{2017}. For the sake of completeness we carry on the details here.
	We first prove $u_0\in C^\infty(\R)$. For that it suffices to show that $e^{-x}u_0\in C^\infty(\R)$. Thus, in view of Sobolev's embedding, it suffices to prove that $\partial_{x}^m(e^{-x}u_0)\in L^2(\R)$, for any $m\in \Z^+$.  To prove this, let us consider the IVP
	\begin{equation*}
		\begin{cases}
			\partial_t w + Lw =0, \quad t<0,\\
			w(x,0)=e^{-x}\phi,
		\end{cases}
	\end{equation*}
	where
	$$Lw=\gamma\partial_x^5 w +5\gamma\partial_x^4 w + (\beta +10\gamma)\partial_x^3 w +(3\beta+10\gamma)\partial_x^2 w + (3\beta+5\gamma)\partial_x w +(\beta+\gamma)w.$$
For one hand, since $\partial_t(e^{x}w)+\beta\partial_x^3 (e^{x}w)+\gamma\partial_x^5(e^{x}w)=0$ and $e^{x}w(x,0)=\phi(x)$ we deduce that $W(t)\phi(x)=e^{x}w(x,t)$. On the other hand, it is easy to see that the solution of the above IVP is
$$
w(x,t)=W(t)e^{-5\gamma t \partial_x^4}e^{-10\gamma t\partial_x^3}e^{-(3\beta+10\gamma)t\partial_x^2}\left(e^{-x+(2\beta+4\gamma)t}\phi(x-(3\beta+5\gamma)t) \right).
$$
Hence,
$$
e^{-x}W(t)\phi(x)=W(t)e^{-5\gamma t \partial_x^4}e^{-10\gamma t\partial_x^3}e^{-(3\beta+10\gamma)t\partial_x^2}\left(e^{-x+(2\beta+4\gamma)t}\phi(x-(3\beta+5\gamma)t) \right).
$$
Next using Plancherel's theorem and the facts that $\gamma<0$ and $3\beta+10\gamma>0$ we deduce
\begin{equation}\label{class}
	\begin{split}
\|\partial_{x}^m(e^{-x}W(t)\phi)\|_{L^2}&\leq \|\xi^me^{(3\beta+10\gamma)t\xi^2}\|_{L^\infty}\|	e^{-x+(2\beta+4\gamma)t}\phi(x-(3\beta+5\gamma)t)\|_{L^2}\\
&	\leq \frac{c_{m} e^{-(\beta+\gamma) t}}{((3 \beta+10 \gamma) |t|)^{k / 2}},
	\end{split}
\end{equation}
where $c_m$ is a constant depending on $m$ and we have used that $e^{-x}\phi\in L^2(\R)$.

Inequality \eqref{class} now yields
\[
\begin{split}
\|\partial_{x}^m(e^{-x}u_0)\|_{L^2}&\leq \sum_{j=1}^\infty\alpha_j\|\partial_{x}^m(e^{-x}W(-\sigma j)\phi)\|_{L^2}\\
&\leq \sum_{j=1}^\infty\alpha_j\frac{c_{m} e^{-(\beta+\gamma) \sigma j}}{((3 \beta+10 \gamma) \sigma j)^{m / 2}}.
\end{split}
\]
By choosing $\alpha_j$ such that the above series converges for any $m\in\Z^+$ (for instance, take $\alpha_j:=e^{-j^2}$).  we conclude that  $u_0\in C^\infty(\R)$.

Since  $W(t)$ is bounded in $H^s(\R)$, the fact $u_0\in {H^{7/2}}^-(\R) \cap L^2({\langle x\rangle^{7/4}}^-dx)$ follows directly from inequality \eqref{trocas} and the properties of $\phi$.

Before  proving (i) and (ii), let us now consider the  IVP 
	\begin{equation*}
	\begin{cases}
		\partial_t w + Lw=0, \quad t>0,\\
	w(x,0)=e^{x}\phi,
	\end{cases}\end{equation*}
	where $$Lw=\gamma\partial_x^5 w -5\gamma\partial_x^4 w + (\beta +10\gamma)\partial_x^3 w -(3\beta+10\gamma)\partial_x^2 w + (3\beta+5\gamma)\partial_x w -(\beta+\gamma)w.$$
Here we have  $W(t)\phi(x)=e^{-x}w(x,t)$ and $w$ is given by the expression
	\begin{equation*}\label{classicccc}
	 w(x,t)=W(t)e^{5\gamma t \partial_x^4}e^{-10\gamma t\partial_x^3}e^{(3\beta+10\gamma)t\partial_x^2}\left(e^{x-(2\beta+4\gamma)t}\phi(x-(3\beta+5\gamma)t) \right).
	\end{equation*}
Thus,
$$
\partial_x^m(e^{x}W(t)\phi(x))=\partial_{x}^m W(t)e^{5\gamma t \partial_x^4}e^{-(2\beta+10\gamma) t\partial_x^3}e^{(3\beta+10\gamma)t\partial_x^2}\left(e^{x-(2\beta+4\gamma)t}\phi(x-(3\beta+5\gamma)t) \right)
$$
with
\begin{equation}\label{class1}
	\begin{split}
\|\partial_x^m(e^{x}W(t)\phi)\|_{L^2}&\leq \|\xi^me^{-(3\beta+10\gamma)t\xi^2}\|_{L^\infty}\|	e^{x-(2\beta+4\gamma)t}\phi(x-(3\beta+5\gamma)t)\|_{L^2}\\
&	\leq \frac{c_{m} e^{(\beta+\gamma) t}}{((3 \beta+10 \gamma) t)^{m / 2}},	
	\end{split}
\end{equation}
where we used that $e^{x}\phi\in L^2(\R)$.

	We now establish conditions (i) and (ii). 
	To see that (i) holds, assume $t>0$ is so that $t\notin \sigma \Z^+$. As before, it is enough to prove $e^{-x}W(t)u_0\in H^m(\R)$ for all $m\in\Z^+$. From \eqref{class}, \eqref{class1} and the fact that $\phi$ is symmetric,  we get
	\begin{equation}
		\begin{split}
	 \|\partial_x^m\left( e^{-x}W(t)u_0\right)\|_{L^2}
	 &\le \sum_{j=1}^{\infty} \alpha_j \left\|\partial_x^m \left(e^{-x} W(t-\sigma j)\phi\right)\right\|_{L^2}\\
	 &\le c_m \sum_{j=1}^{\infty}\alpha_j \frac{e^{(\beta+\gamma)|t-\sigma j|}}{((3\beta+10\gamma)|t-\sigma j|)^{m/2}}. \label{rightmost}
	\end{split}
 	\end{equation}
By our choice of $\alpha_j$, the rightmost series in \eqref{rightmost} is finite for all $m\in \mathbb{Z}^+$.

	Finally, to prove (ii), assume $t=\sigma n$, for some $n\in\Z^+$. We have
	$$W(t)u_0=\alpha_n \phi + \sum_{\substack {j=1\\ j\neq n}}^{\infty}\alpha_jW(\sigma(n- j))\phi.$$ Using the above arguments, we may show that the series belongs to $C^\infty(\R)$. 	The conclusion then follows because $\phi\in C^3(\R\setminus\{0\})\setminus C^3(\R)$.
\end{proof}

\subsubsection{Nonlinear smoothing}
The goal of this section is to prove that the integral term in the Duhamel formulation of the solution of \eqref{kawaivp} is more regular than the solution of the corresponding linear equation.

We begin by recalling some useful inequalities.

\begin{lemma}
Assume $T\in (0,1)$ and let $W(t)$ be as in \eqref{necesidade}. 
\begin{itemize}
	\item[(i)] For any $\varphi\in L^2(\R)$,
\begin{equation}\label{st1}
			\|D^2W(t)\varphi\|_{L^\infty_xL^2_T}\le C\|\varphi\|_{L^2_x}.
\end{equation}
\item[(ii)] If $f \in L^1_TL^2_x$ then
\begin{equation}\label{dualkawa}
	\sup_{[0,T]}\left\| D^2\int_{0}^{t}W(t-t')f(\cdot,t')dt' \right\|_{L^2_x} \le C \|f\|_{L^1_xL^2_T}. 
\end{equation}
\item[(iii)] For any $\theta\in(0,1)$, $-1<\alpha\le \frac{3}{2}$ and $\varphi\in L^2(\R)$,
\begin{equation}\label{str2.4}
	\|D^{\frac{\theta \alpha}{2}}W(t)\varphi\|_{L^q_T L^p_x}\le C \|\varphi\|_{L^2_x},
\end{equation}
where $p=2/(1-\theta)$ and $q=10/\theta (\alpha +1)$.
\end{itemize}
\end{lemma}
\begin{proof}
	For \eqref{st1} see \cite[Theorem 2.6]{KAWA}. Estimate \eqref{dualkawa} follows from \eqref{st1} and a duality argument. For \eqref{str2.4} see \cite[Theorem 2.4]{KAWA}.
\end{proof}

With the above inequalities in hand we are able to prove the following result.

\begin{proposition}\label{kawaduh}
	Let $\frac{13}{6}<s<4$ and assume $u_0\in H^s(\R)\cap L^2(|x|^{s/2}dx)$. Let $u(t)$ be the solution of the  IVP \eqref{kawaivp} provided by Theorem \ref{persistkawa},
	\begin{equation}
		u(t)=W(t)u_0+\int_0^t W(t-t')(u\partial_x u)dt'=:W(t)u_0+\mathcal{Z}(t), \quad t\in[0,T].
	\end{equation}  
	Then, $\mathcal{Z}(t)\in H^{s+1}(\R)$ for any $t\in[0,T]$.
\end{proposition}
\begin{proof}
From Theorem \ref{persistkawa} we already know that $\mathcal{Z}(t)\in H^s(\R)$. So we only need to prove that $D^{s+1}\mathcal{Z}(t)\in L^2(\R)$. Note that in the proof of Theorem \ref{persistkawa} we have assumed $0<T<1$; thus,	in view of \eqref{dualkawa} we have
	\begin{align*}
	\left\| D^{s+1}\int_{0}^{t}W(t-t')(u\partial_x u)dt' \right\|_{L^2_x} &\le C \|D^{s-1}(u\partial_x u)\|_{L^1_xL^2_T}\\ &\le C\left(\|uD^{s-1}\partial_x u\|_{L^1_xL^2_T}+\left\|\left[D^{s-1},u \right]\partial_x u\right\|_{L^1_xL^2_T}\right)\\&\equiv C \left(\RNum{1}+\RNum{2}\right), 
	\end{align*}
	where $\left[D^{s-1},u \right]\partial_x u=D^{s-1}(u\partial_xu)-uD^{s-1}\partial_xu$.
	According to H\"older's inequality,
	$$
	\RNum{1}\le \|u\|_{L^{6/5}_x L^3_T} \|D^{s-1}\partial_x u\|_{L^6_{xT}}\equiv \RNum{1}_1 \RNum{1}_2. $$
	
	In order to estimate $\RNum{1}_1$ we use H\"older's inequality again to get 
	$$
	\RNum{1}_1\le \|\langle x\rangle^{-r}\|_{L^2_x}\|\langle x\rangle^{r} u\|_{L^3_T L^3_x}\le CT^{1/3}\|\langle x\rangle^{r}u\|_{L^\infty_T L^3_x}, 
	 $$ 
	where $r>\frac{1}{2}$. Using the embedding $H^{1/6}(\R)\hookrightarrow L^3(\R)$ together with the interpolation \eqref{interp2} we get 
	\begin{equation}\label{gotta}
		I_1\le C T^{1/3}\|J^{1/6}\langle x\rangle^{r}u\|_{L^\infty_T L^2_x}\le C T^{1/3}\|J^s u\|_{L^\infty_TL^2_x}^{\theta}\|\langle x\rangle^{s/4}u\|_{L^\infty_T L^2_x}^{1-\theta}, 
	\end{equation}  
where $\theta=1/6s$. Note that to apply \eqref{interp2} we have written $r=(1-\theta)\frac{s}{4}$, therefore the condition  $r>\frac{1}{2}$ forces $s>\frac{13}{6}$. According to the weighted local theory the right-hand side of \eqref{gotta} is finite.
	
	On the other hand, from H\"older's inequality and the Strichartz estimate \eqref{str2.4} with $\theta=\frac{2}{3}$ and $\alpha=0$ it follows that
	\begin{equation}\label{gotta1}
		\begin{split}
		\RNum{1}_2&\le T^{1/10}\|D^{s-1}\partial_x u\|_{L^{15}_TL^6_x}\le T^{1/10}\|D^{s-1}\partial_x u_0 \|_{L^{15}_TL^6_x}\\  &\hspace{50mm} +T^{1/10}\left\|D^{s-1}\partial_x W(t)\int_{0}^{t}W(-t')u\partial_x u dt' \right\|_{L^{15}_TL^6_x}\\ &\le CT^{1/10}\|D^s u_0\|_{L^2_x}+CT^{1/10}\int_0^T\left\| D^s(u\partial_x u)\right\|_{L^2_x} dt'\\&\le
		CT^{1/10}\|D^s u_0\|_{L^2_x}+CT^{3/5}\|D^s(u\partial_x u)\|_{L^2_{xT}},
		\end{split}
		\end{equation}
	which is finite according to Lemma \ref{lemonkawa}.
	
	From \eqref{gotta} and \eqref{gotta1} we conclude that $I$ is finite.
	
	It remains to prove $\RNum{2}$ is finite. For that let us introduce the weights  $v=w=\langle x\rangle^{r}$, with $r>1$ to be determined latter. By setting $\ell=2$ and $p=q=4$ we see that $v^{\frac{\ell}{p}}w^{\frac{\ell}{q}}=\langle x\rangle^{r}$. Since $\langle x\rangle^{r}\in A_4$, from H\"older's inequality and Lemma \ref{commuter} we obtain
	\begin{equation}\label{gotta2}
		\begin{split}
		II&\leq C \|\langle x\rangle^{-r/2}\|_{L^2_x}\|\langle            x\rangle^{r/2}[D^{s-1},u]\partial_x u\|_{L^2_{xT}}\\
		&\leq  C\left\|\|\langle x\rangle^{r/4}D^{s-1}u\|_{L^4_x} \|\langle x\rangle^{r/4}\partial_x u\|_{L^4_x} +\|\langle x\rangle^{r/4}\partial_x u\|_{L^4_x} \|\langle x\rangle^{r/4}D^{s-2}\partial_xu\|_{L^4_x} \right\|_{L^2_T}.
		\end{split}
	\end{equation}
Because $D=\mathcal{H}\partial_x$ we infer from Lemma \ref{hilb} that
$$
\|\langle x\rangle^{r/4}D^{s-2}\partial_xu\|_{L^4_x}=\|\langle x\rangle^{r/4}\mathcal{H}D^{s-1}u\|_{L^4_x}\leq C\|\langle x\rangle^{r/4}D^{s-1}u\|_{L^4_x},
$$
which, from \eqref{gotta2}, yields
\[
\begin{split}
II&\leq C\left\|\|\langle x\rangle^{r/4}D^{s-1}u\|_{L^4_x} \|\langle x\rangle^{r/4}\partial_x u\|_{L^4_x} \right\|_{L^2_T}\\
&\leq C \left\|\|\langle x\rangle^{r/4}D^{s-1}u\|_{L^4_x}^2  \right\|_{L^2_T} + C\left\|\|\langle x\rangle^{r/4}\partial_x u\|_{L^4_x}^2  \right\|_{L^2_T}\\
&\leq C\left\|\|\langle x\rangle^{r/4}D^{s-1}u\|_{L^4_x}  \right\|_{L^4_T}^2 + C\left\|\|\langle x\rangle^{r/4}\partial_x u\|_{L^4_x}  \right\|_{L^4_T}^2 \equiv C\RNum{2}_1+C\RNum{2}_2.
\end{split}
\]
	We begin estimating $\RNum{2}_1$ by using \eqref{interpp} (with $D$ instead of $J$):
	\begin{equation} \label{gotta2.1}
		\begin{split}
		\RNum{2}_1^{1/2}&=\left\| \|\langle x\rangle^{r/4}D^{s-1}u\|_{L^4_x}  \right\|_{L^4_T}\le C \left\| \|\langle x\rangle^{b}u\|_{L^4_x}^{\theta} \|D^a u\|_{L^4_x}^{1-\theta} \right\|_{L^4_T}\\ 
		&\le C \left\| \|\langle x\rangle^{b}u\|_{L^4_x} + \|D^a u\|_{L^4_x} \right\|_{L^4_T}\\
		&\le C T^{1/4} \|\langle x\rangle^{b}u\|_{L^\infty_TL^4_x} + CT^{1/8}\|D^a u\|_{L^8_TL^4_x} 
	\end{split}
	\end{equation}
	where 
	\begin{equation}\label{gotta3}
	\theta\in(0,1), \quad a=\frac{s-1}{1-\theta}, \quad \mbox{and}\quad b=\frac{r}{4\theta}.
	\end{equation}
	For the term $\|\langle x\rangle^{b} u\|_{L^\infty_TL^4_x}$ we use the embedding $H^{1/4}(\R)\hookrightarrow L^4(\R)$ and \eqref{interp2} to obtain
	\begin{equation}\label{gotta3.1}
		\|\langle x\rangle^{b} u\|_{L^\infty_TL^4_x}\le C \|J^{1/4}(\langle x\rangle^{b} u)\|_{L^\infty_T L^2_x} \le C \|J^s u\|_{L^\infty_T L^2_x}^{1-\lambda} \|\langle x\rangle^{s/4} u\|_{L^\infty_T L^2_x}^\lambda <\infty,
	\end{equation}
	with
	\begin{equation}\label{gotta4}
\lambda\in(0,1), \quad	\lambda s=\frac{1}{4}, \quad \mbox{and}\quad  \quad(1-\lambda)\frac{s}{4}=b.
	\end{equation}
Conditions \eqref{gotta3} and \eqref{gotta4} leave $\theta=\frac{4r}{4s-1}$, which is in the interval $(0,1)$ provided $s>r+\frac{1}{4}$. Hence, if  $r=1+\varepsilon$ for some $0<\varepsilon<11/12$ we see that \eqref{gotta3.1} holds for any $s>13/6$.

For the second term on the right-hand side of \eqref{gotta2.1}, according to the choice of $\theta$ and $r$ above, we have $$
a=\frac{s-1}{1-\theta}=\frac{4s^2-5s+1}{4s-(5+4\varepsilon)}=s+\varepsilon+\frac{1}{4s-5-4\varepsilon}+\frac{\varepsilon(5+4\varepsilon)}{4s-5-4\varepsilon}.
$$
Therefore, by assuming $\varepsilon$ sufficiently small we may write
$$
a=s+\frac{1}{4s-5}+\tilde{\varepsilon}
$$
 where $\tilde{\varepsilon}>0$ is also small enough. By setting $\delta=\frac{3}{8}-\tilde{\varepsilon}-\frac{1}{4s-5}$, our assumption $s>\frac{13}{6}$ gives $\delta>0$ and we may write $a=3/8+s-\delta$. We employ the Strichartz estimate \eqref{str2.4} with $\theta=\frac{1}{2}$ and $\alpha=\frac{3}{2}$ to get
	\begin{equation}\label{gotta5}
		\begin{split}
	\|D^a u\|_{L^8_TL^4_x}&\le \|D^{3/8}W(t)D^{s-\delta} u_0\|_{L^8_TL^4_x}+\left\|D^{3/8}W(t)\int_0^tW(-t')D^{s-\delta}(u\partial_x u)dt'\right\|_{L^8_TL^4_x}\\&\le \|D^{s-\delta}u_0\|_{L^2_x}+\int_0^T\|D^{s-\delta} u\partial_x u\|_{L^2_x}dt' \\ &\le \|u_0\|_{H^s} + CT^{1/2}\|u\partial_x u\|_{L^2_TH^s}<\infty ,
		\end{split}
		\end{equation}
where the right-hand side of the above inequality is finite thanks to Lemma \ref{lemonkawa}.  This proves $\RNum{2}_1<\infty$.
	
	To see that $\RNum{2}_2$ is finite we proceed in  exactly the same manner by noticing that $\RNum{2}_2$ is almost the same as $\RNum{2}_1$ but with less derivatives. Indeed, from Lemma \ref{hilb} and \eqref{commuter},
	\[
	\begin{split}
\RNum{2}_2^{1/2}&= \|\langle x\rangle^{r/4}\mathcal{H}D u\|_{L^4_{xT}}
\leq C\|\langle x\rangle^{r/4}D u\|_{L^4_{xT}}\\
&\leq CT^{1/4}\|\langle x\rangle^{r/4\theta}u\|_{L^\infty_TL^4_x} + CT^{1/8}\|D^{\frac{1}{1-\theta}}u\|_{L^8_T L^4_x}
	\end{split}
	\] 
with (as in \eqref{gotta3.1})
 $$\|\langle x \rangle^{r/4\theta}u\|_{L^\infty_TL^4_x}\le C \|J^{1/4}(\langle x\rangle^{r/4\theta}u)\|_{L^\infty_TL^2_x}\le C\|J^s u\|_{L^\infty_TL^2_x}^{1-\lambda} \|\langle x\rangle^{s/4}u\|_{L^\infty_T L^2_x}^\lambda<\infty. $$
Besides, since
$$
\frac{1}{1-\theta}=\frac{4s-1}{4s-5-4\varepsilon}=1+\frac{4}{4s-5}+\tilde{\tilde{\varepsilon}}=\frac{3}{8}+\eta,
$$
for $\tilde{\tilde{\varepsilon}}>0$ small and $\eta=\frac{5}{8}+\frac{4}{4s-5}+\tilde{\tilde{\varepsilon}}<s$, as done in \eqref{gotta5} we deduce
	\begin{align*}\nonumber \|D^{\frac{1}{1-\theta}}u\|_{L^8_T L^4_x}&= \|D^{3/8}D^\eta u \|_{L^8_TL^4_x}\\
		&\le C\|u_0\|_{H^\eta}+CT^{1/2}\|u\partial_x u\|_{L^2_TH^\eta}\\
		&\le C\|u_0\|_{H^s}+CT^{1/2}\|u\partial_x u\|_{L^2_TH^s}<\infty.
	\end{align*}
This shows that $II_2$ is finite and completes the proof of the proposition.
\end{proof}

\begin{proof}[Proof of Theorem \ref{dispkawa}]
	Let $u_0\in C^\infty(\R)\cap{H^{7/2}}^-(\R) \cap L^2({\langle x\rangle^{7/4}}^-dx)$ be the initial data constructed in Proposition \ref{step1kawa}. Let $u(t)$, $t\in[0,T]$, be the solution corresponding. We may assume that $\sigma$ is sufficiently small such that $\sigma\in (0,T)$. Thus, for $t^*=\sigma$,
	$$
	u(t^*)=W(t^*)u_0+\int_0^{t^*}W(t^*-t')(u\partial_x u)dt'=:W(t^*)u_0+\mathcal{Z}(t^*).
	$$
	From Proposition \ref{kawaduh} we know that $\mathcal{Z}(t^*)\in H^{\frac{9}{2}^-}(\R)\hookrightarrow C^3(\R)$. Since $W(t^*)u_0\in C^3(\R\setminus\{0\})\setminus C^3(\R)$, the conclusion then follows from Proposition \ref{step1kawa}.
\end{proof}

\subsection{The Hirota-Satsuma system}

Here we prove Theorem \ref{disphss} in the same spirit of Section \ref{KEsec}. So, we first construct an appropriate initial data for the corresponding linear problem and then show that the integral part of the solution is smoother than the linear one.

\subsubsection{Construction of the initial data}
Let $\{U_a(t)\}$ and $\{U(t)\}$ be the unitary groups introduced in Section \ref{sec3.1}. In
 \cite[Section 3]{2017} the authors showed that, for some suitable sequence $\{\alpha_j\}$,
 $$
 w_0(x):=\sum_{j=1}^{\infty}\alpha_jU(-j)e^{-2|x|}
 $$
 belongs to 
 $$
 C^{\infty}(\R)\cap L^\infty(\R)\cap H^{3/2^-}(\R)\cap L^2(\langle x\rangle^{3/2^-}dx)
 $$
 and satisfies:
 \begin{itemize}
 	\item[(i)] for any $t\notin \mathbb{Z}$, $U(t)w_0\in C^1(\R)$;
 	\item[(ii)] for any $t\in \mathbb{Z}$, $U(t)w_0\in C^1(\R\setminus\{0\})\setminus C^1(\R)$.
 \end{itemize}
Here, with a slightly modification of their proof and by taking 
 \begin{equation}\label{hirota1}
  	u_0(x):=\sum_{j=1}^{\infty}\alpha_j U_a(-\sigma j)e^{-2|x|}\hspace{5mm}\mbox{and}\hspace{5mm} v_0(x):=\sum_{j=1}^{\infty}\alpha_j U(-\sigma j)e^{-2|x|},
  \end{equation}
  for some real constant $\sigma$, we can show the following.

\begin{proposition}\label{step1hss}
	The functions in \eqref{hirota1} satisfy
	$$(u_0,v_0)\in \left(C^\infty(\R)\cap{H^{3/2}}^-(\R) \cap L^2({\langle x\rangle^{3/2}}^-dx)\right)^2.$$  Moreover, the associated global-in-time solution $(u,v)\in \left(C(\R;{H^{3/2}}^-(\R))\right)^2$ of the linear part of the IVP \eqref{system} satisfy
	\begin{enumerate}
		\item[(i)] For any $t>0$ with $t\notin \sigma\Z^+$ we have $(u,v)(\cdot,t)\in \left(C^\infty(\R)\right)^2.$
		\item[(ii)] For any $t\in\sigma\Z^+$ we have $(u,v)(\cdot,t)\in \left(C^1(\R\setminus\{0\})\setminus C^1(\R)\right)^2$. 
	\end{enumerate}
\end{proposition}
\begin{proof}
	See Section 3 in  \cite{2017} (see also Lemma 3.2 in \cite{palacios}).
\end{proof}

\subsubsection{Nonlinear smoothing}

Let us start by recalling some linear estimates.

\begin{lemma}\label{kpvlemma}
For any $a\neq0$ and $u_0\in L^2(\R)$ we have
\begin{equation}\label{ww1}
 \|D^{-1/4}_xU_a(t)u_0\|_{L^4_xL^\infty_T}\le C_a\|u_0\|_{L^2},
\end{equation}
\begin{equation}\label{ww2}
\|\partial_x U_a(t)u_0\|_{L^\infty_x L^2_T}\le C_a\|u_0\|_{L^2}
\end{equation}
and
\begin{equation}\label{wildstrich}
	\|D^{-1/12}_x U_a(t)u_0\|_{L^{60/13}_x L^{15}_T}\le C_a\| u_0\|_{L^2}.
\end{equation}
\end{lemma}
\begin{proof}
For \eqref{ww1} and \eqref{ww2} see Theorems 3.5 and 3.7 in \cite{KPV}. Estimate \eqref{wildstrich} follows interpolating \eqref{ww1} and \eqref{ww2}; indeed, it suffices to define the family of analytic operators $T_zu_0=D^{z/4}D^{1-z}U_a(t)u_0$, $0\leq \textrm{Re}(z)\leq 1$ and apply the Stein interpolation theorem with $z=\frac{13}{15}$ (see a similar result in Corollary 3.8 of \cite{KPV}).
\end{proof}

We also recall the following Strichartz estimate :

\begin{lemma} For any $a\neq0$ and $u_0\in L^2(\R)$,
\begin{equation}\label{stric}
	\|D^{\alpha \theta/2}U_a(t)u_0\|_{L^q_TL^p_x}\le C\|u_0\|_{L^2},
\end{equation}
where $(q,p)=\left(\frac{6}{\theta(\alpha +1)}, \frac{2}{1-\theta}\right)$ and $(\theta,\alpha)\in(0,1)\times[0,1/2]$.
\end{lemma}
\begin{proof}
See Lemma 2.4 in \cite{KPV2}.
\end{proof}

Below we also need to use identity \eqref{trocaflp}; so, we recall the precise estimate for the term $\Phi_{t,\alpha}$.

\begin{lemma}
	Let $\alpha\in(0,1)$. If $u_0\in H^{2\alpha}(\R)\cap L^2(|x|^{2\alpha}dx)$ then the identity
	\begin{equation}\label{troca2}
		|x|^\alpha U_a(t)u_0=U_a(t)(|x|^\alpha u_0)+\left[\Phi_{t,\alpha}(\widehat{u_0}(\xi))\right]^\vee(x)
	\end{equation}
	 holds for any $t\in\R$ and almost $x\in \R$ with
	$$
	\|\left[\Phi_{t,\alpha}(\widehat{u_0}(\xi))\right]^\vee\|_{L^2}\leq C(1+|t|)\left(\|u_0\|_{L^2}+\|D^{2\alpha}u_0\|_{L^2}\right).
	$$ 
\end{lemma}
\begin{proof}
See Theorem 1 in \cite{FLP}.
\end{proof}

 With these tools in hand we can prove the following smoothing property.

\begin{proposition}\label{hssduh}
	Let $\frac{7}{6}<s<\frac{11}{6}$ and consider an initial data $(u_0,v_0)\in \left(H^s(\R)\cap L^2(|x|^{s}dx)\right)^2$. Let $(u,v)(t)$ be the solution of the Hirota-Satsuma system \eqref{system} provided by Theorem \ref{persistHSS} and given by \begin{equation*}
		\left\{	\begin{array}{l}
			u(t)=U_a(t)u_0+\displaystyle\int_{0}^{t}U_a(t-t^\prime)(6au\partial_xu - 2rv\partial_x v)(t^\prime)dt^\prime:=U_a(t)u_0 + \mathcal{Z}_1(t) \\
			v(t)=U(t)v_0+3\displaystyle\int_{0}^{t}U(t-t^\prime)(u\partial_x v)(t^\prime)dt^\prime:=U(t)v_0 + \mathcal{Z}_2(t).
		\end{array}\right.
	\end{equation*}
	Then, $\mathcal{Z}_i(t)\in H^{s+\frac{1}{6}}(\R)$, $i=1,2$, $t\in[0,T]$.
\end{proposition}
\begin{proof}

We show the computations for $\mathcal{Z}_2$, same procedure apply to $\mathcal{Z}_1$. Since $\mathcal{Z}_2(t)\in H^s(\R)$ it suffices to show that $\|D^{s+\frac{1}{6}}\mathcal{Z}_2(t)\|_{L^2_x}$ is finite. We follow partially the ideas in \cite[Lemma 5.2]{palacios}. For that, first note that \eqref{ww2} and duality give
\begin{equation}\label{dualhh}
	\sup_{[0,T]}\left\| \partial_x\int_{0}^{t}U(t-t')f(\cdot,t')dt' \right\|_{L^2_x} \le C \|f\|_{L^1_xL^2_T}. 
\end{equation}
Hence, using \eqref{dualhh}, Lemma \ref{A12} (part (iii)), and H\"older's inequality we infer
\begin{flalign*}
	\left\| D^{s+\frac{1}{6}}\mathcal{Z}_2(t) \right\|_{L^2_x}&\le C \|D^{s-\frac{5}{6}}(u\partial_x v)\|_{L^1_xL^2_T}\\ &\hspace{4mm} \le C\|D^{s-\frac{5}{6}}(u\partial_x v)-uD^{s-\frac{5}{6}}\partial_x v-\partial_x v D^{s-\frac{5}{6}}u\|_{L^1_xL^2_T}\\&\hspace{9mm}+C\|uD^{s-\frac{5}{6}}\partial_x v \|_{L^1_xL^2_T}+C\|\partial_x vD^{s-\frac{5}{6}}u \|_{L^1_xL^2_T}\\&\hspace{4mm} \le C \|u\|_{L^{6/5}_xL^3_T}\|D^{s+\frac{1}{6}}v\|_{L^6_{xT}}+C\|\partial_x v\|_{L^{60/13}_xL^{15}_T}\|D^{s-\frac{5}{6}}u \|_{L^{60/47}_xL^{30/13}_T}\\ &\hspace{4mm} \le C\{ \RNum{1}+\RNum{2}_1\RNum{2}_2\}.
\end{flalign*}

To see  that $\RNum{1}$ is finite we combine the ideas developed in Section \ref{sec3.1}  together with \cite{2017} and the proof of Lemma 5.2 in \cite{palacios}. Indeed, using \eqref{stric} with $\alpha=\frac{1}{2}$, $\theta=\frac{2}{3}$ and $p=q=6$, we obtain
\[
\begin{split}
\|D^{s+\frac{1}{6}}v\|_{L^6_{xT}}&\leq \|D^{\frac{1}{6}}U(t)D^sv_0\|_{L^6_{xT}}+\left\|D^{\frac{1}{6}}U(t)\int_0^tU(-t')D^s(u\partial_x v)(t^\prime)dt^\prime\right\|_{L^6_{xT}}\\
&\leq C\|D^su_0\|_{L^2_x}+C\int_0^T\|u\partial_x v\|_{L^2_x}dt'\\
&\leq C\|u_0\|_{s,2}+CT^{1/2}\|u\partial_x v\|_{L^2_TH^s}.
\end{split}
\]
The last term in the above inequality has already been shown to be finite in the local theory (see for instance \eqref{tele0}). This shows that $\|D^{s+\frac{1}{6}}v\|_{L^6_{xT}}$ is finite. To see that $\|u\|_{L^{6/5}_xL^3_T}$ is finite we need to use the local theory in weighted spaces. In fact, from H\"older's inequality, Sobolev embedding and \eqref{interp2} we deduce, for some $r=\frac{1}{2}^+$,
\[
\begin{split}
\|u\|_{L^{6/5}_xL^3_T}&\leq C\|\langle x\rangle^ru\|_{L^3_{xT}}
\leq CT^{1/3}\|\langle x\rangle^ru\|_{L^\infty_TL^3_{x}}\\
&\leq C\|J^{1/6}(\langle x\rangle^ru)\|_{L^\infty_TL^2_{x}}\\
&\leq C \|J^su\|_{L^\infty_TL^2_{x}}^{1-\lambda} \|\langle x\rangle^{s/2^-}u\|_{L^\infty_TL^2_{x}}^{\lambda},
\end{split}
\]
with $\lambda\frac{s}{2}^-=r$ and $\frac{1}{6}<(1-\lambda)s$. Since $s>7/6$ we may take $\lambda=\frac{1}{s}^+$ to conclude that $I$ is finite.

In view of \eqref{wildstrich}, 
\begin{align*}
	&\RNum{2}_1\le\|\partial_x U(t)v_0\|_{L^{60/13}_xL^{15}_T}+\left\|\partial_x U(t)\int_0^t U(-t')u\partial_x v dt' \right\|_{L^{60/13}_xL^{15}_T}\\
	&\hspace{5mm}\le C\|\mathcal{H}D^{13/12}v_0\|_{L^2}+C\int_0^T\left\|\mathcal{H} D^{13/12}(u\partial_x v) \right\|_{L^2}dt'\\
	&\hspace{5mm}\le C\|v_0\|_{s,2}+CT^{1/2}\left\|u\partial_x v \right\|_{L^2_TH^s}<\infty,
\end{align*}
where in the last inequality we used that $\mathcal{H}$ is bounded in $L^2$ and the fact that $s>\frac{7}{6}>\frac{13}{12}$. Again, the term $\left\|u\partial_x v \right\|_{L^2_TH^s}$ may be bounded as done in the local theory.

In what comes to $\RNum{2}_2$ we argue as follows. For $\gamma>7/20$ (to be chosen latter) we have $$
II_2\le \|\langle x\rangle^{-\gamma}\|_{L_x^{20/7}}\|\langle x\rangle^{\gamma}D^{s-\frac{5}{6}}u\|_{L^{30/13}_{xT}}\leq C\|\langle x\rangle^{\gamma}D^{s-\frac{5}{6}}u\|_{L^{30/13}_{xT}}.
$$ 
Set $\mathcal{\dot{Z}}_1(t)=U_a(-t)\mathcal{Z}_1(t)$.  Using  H\"older's inequality in time and \eqref{troca2} we get
\begin{align*}
	&\RNum{2}_2\le CT^{37/90}\left\{\left\|U_a(t)\left(\langle x\rangle^\gamma D^{s-\frac{5}{6}}u_0\right)\right\|_{L^{45}_TL^{30/13}_x} + \left\|U_a(t)\{\Phi_{t,\gamma}\widehat{D^{s-\frac{5}{6}}u_0}\}^\vee\right\|_{L^{45}_TL^{30/13}_x}  \right. \\ &\hspace{24mm} \left. +\left\|U_a(t)\left(\langle x\rangle^\gamma D^{s-\frac{5}{6}}\mathcal{\dot{Z}}_1\right)\right\|_{L^{45}_TL^{30/13}_x} + \left\|U_a(t)\{\Phi_{t,\gamma}\widehat{D^{s-\frac{5}{6}}\mathcal{\dot{Z}}_1}\}^\vee\right\|_{L^{45}_TL^{30/13}_x}   \right\}.
\end{align*}
Next, by setting $\gamma=5/12$ and using Strichartz estimate \eqref{stric} with $\alpha=0$ and $\theta=2/15$ we deduce
\begin{equation}\label{sign}
\begin{split}
	\hspace{1mm}	&\RNum{2}_2\le CT^{37/90}(1+T)\left\{\|\langle x\rangle^\gamma D^{s-\frac{5}{6}}u_0\|_{L^2} + \|D^{2\gamma+s-\frac{5}{6}}u_0 \|_{L^2}+ \|D^{s-\frac{5}{6}}u_0 \|_{L^2}  \right. \\ &\hspace{35mm} \left. +\|\langle x\rangle^\gamma D^{s-\frac{5}{6}}\mathcal{\dot{Z}}_1\|_{L^2} + \|D^{2\gamma+s-\frac{5}{6}}\mathcal{\dot{Z}}_1\|_{L^2} + \|D^{s-\frac{5}{6}}\mathcal{\dot{Z}}_1\|_{L^2}  \right\}\\	&\hspace{5mm}\le CT^{37/90}(1+T)\left\{\|\langle x\rangle^\gamma D^{s-\frac{5}{6}}u_0\|_{L^2} +\|\langle x\rangle^\gamma D^{s-\frac{5}{6}}\mathcal{\dot{Z}}_1\|_{L^2} +  \|u_0\|_{H^s}+\|\mathcal{\dot{Z}}_1\|_{H^s}   \right\}.\\
\end{split}
\end{equation}
Since
$$
\|\mathcal{\dot{Z}}_1\|_{H^s} \leq C\int_0^T\|(6au\partial_xu - 2rv\partial_x v)\|_{H^s}dt',
$$
we can prove that $\|\mathcal{\dot{Z}}_1\|_{H^s}$ is finite in a similar fashion as done in the local theory. Therefore, to conclude $\RNum{2}_2$ is finite it only remains to bound the first two terms on the right-hand side of \eqref{sign}, which can be estimated using \eqref{interpp} and the weighted local theory. In fact, first note that from \eqref{interpp},
\begin{equation*}
	\|\langle x \rangle^\gamma D^{s-\frac{5}{6}}u_0\|_{L^2}\le C\|\langle x \rangle^{s/2}u_0\|_{L^2}^{1-\lambda}\|D^su_0\|_{L^2}^{\lambda}<\infty,
\end{equation*}
where $\lambda=\frac{6s-5}{6s}$ and $(1-\lambda)\frac{s}{2}=\gamma=\frac{5}{12}$. Also, setting $N(u,v)=6a u\partial_x u -2 r v\partial_x v$ and using \eqref{trocas} we have
\begin{equation*}
	\begin{array}{l}
		\|\langle x\rangle^\gamma D^{s-\frac{5}{6}}\mathcal{\dot{Z}}_1\|_{L^2}\le \displaystyle\int_0^T \|\langle x \rangle^\gamma U_a(-t')D^{s-\frac{5}{6}}N(u,v) \|_{L^2}dt'\vspace{2.3mm}\\\hspace{5mm}\le C(1+T)\displaystyle\int_0^T\|\langle x \rangle^\gamma D^{s-\frac{5}{6}}N(u,v)\|_{L^2}+\|N(u,v)\|_{H^s} dt'\vspace{2.3mm}\\\hspace{5mm}\le C(1+T)\left\{\|\langle x \rangle^\gamma D^{s-\frac{5}{6}}N(u,v)\|_{L^1_TL^2_x} +T^{1/2}\|N(u,v)\|_{L^2_TH^s}\right\}.
	\end{array}
\end{equation*}
The second term in the right-hand side of the above inequality can be bounded as it was done in \eqref{tele0}. For the first term, note that for $t\in [0,T]$ and $\lambda=\frac{6s-5}{6s}$ defined above, we have
\begin{align*} 
	\nonumber	\|\langle x \rangle^\gamma D^{s-\frac{5}{6}}N(u,v)\|_{L^2_x}&\le C\|\langle x \rangle^{s/2}N(u,v)\|_{L^2_x}^{1-\lambda} \|D^sN(u,v)\|_{L^2_x}^{\lambda}\\&\le C\left(\|\langle x \rangle^{s/2}N(u,v)\|_{L^2_x}+\|D^sN(u,v)\|_{L^2_x}\right).
\end{align*}
Hence, 
\begin{align*}
	\|\langle x \rangle^\gamma D^{s-\frac{5}{6}}N(u,v)\|_{L^1_TL^2_x}&\le C(1+T	)T^{1/2}\left\{\|\langle x\rangle^{s/2}N(u,v)\|_{L^2_TL^2_x}+\|D^{s}N(u,v)\|_{L^2_TL^2_x}\right\},
\end{align*}
where both terms, $\|\langle x\rangle^{s/2}N(u,v)\|_{L^2_TL^2_x}$ and $\|D^{s}N(u,v)\|_{L^2_TL^2_x}$, can be estimated using the local theory in weighted spaces as in \eqref{localt}. This completes the proof of the Proposition.
\end{proof}

With Proposition \ref{hssduh} in hand, following the same idea as in the proof of Theorem \ref{dispkawa}, we can prove Theorem \ref{disphss}; so we omit the details.

\section*{Acknowledgement}

This study was financed in part by the Coordenação de Aperfeiçoamento de Pessoal de Nível Superior - Brasil (CAPES) - Finance Code 001. A.P. is partially supported by Conselho Nacional de Desenvolvimento Científico e Tecnológico - Brasil (CNPq) grant 2019/02512-5.

\end{document}